\documentclass[11pt]{amsart}
\usepackage{geometry}                
\usepackage{graphicx}
\usepackage{amssymb}
\usepackage{epstopdf}
\usepackage{color}
\usepackage[all,cmtip]{xy}
\usepackage{hyperref}
\usepackage{siunitx}

\theoremstyle{plain}
\newtheorem{thm}{Theorem}[section]
\newtheorem{lemma}[thm]{Lemma}
\newtheorem{prop}[thm]{Proposition}
\newtheorem{coro}[thm]{Corollary}

\newtheorem{construction}[thm]{Construction}
\newtheorem{assumption}[thm]{Assumption}
\theoremstyle{definition}

\newtheorem{definition}[thm]{Definition}
\newtheorem{remark}[thm]{Remark}

\theoremstyle{remark}

\numberwithin{equation}{section} 	

\newcommand{\R}{\mathbb{R}}
\newcommand{\Z}{\mathbb{Z}}
\newcommand{\C}{\mathbb{C}}

\newcommand{\K}{\mathbb{K}}

\newcommand{\La}{\Lambda}
\newcommand{\la}{\lambda}
\newcommand{\p}{\varphi}
\newcommand{\e}{\varepsilon}
\newcommand{\dd}{\partial}
\newcommand{\ip}{\, \lrcorner \,}
\newcommand{\es}{\varnothing}
\newcommand{\op}{\operatorname}
\newcommand{\sse}{\subseteq}
\newcommand{\x}{\times}

\newcommand{\wt}{\widetilde}
\newcommand{\wh}{\widehat}

\newcommand{\std}{\text{std}}
\newcommand{\bs}{\text{BS}}
\newcommand{\as}{\text{AS}}

\newcommand{\sss}{\vspace{2.5 mm}}
\newcommand{\bcs}{\natural}
\newcommand{\ttimes}{\boxtimes}
\newcommand{\Int}{\operatorname{Int}}
\newcommand{\sbf}{\text{SF}}
\newcommand{\flex}{\text{Flex}}
\newcommand{\bdy}{\partial}
\newcommand{\D}{\mathbb{D}}
\newcommand{\calL}{\mathcal{L}}

\newcommand{\Op}{\mathcal{O}p\,}

\newcommand{\calM}{\mathcal{M}}

\newcommand{\sh}{SH}

\newcommand{\fL}{\mathcal{L}}
\newcommand{\fW}{\mathcal{W}}
\newcommand{\fS}{\mathcal{S}}
\newcommand{\fLhat}{\wh{\mathcal{L}}}

\newcommand{\fShat}{\wh{\mathcal{S}}}

\newcommand{\can}{\text{can}}
\newcommand{\oo}{\mathfrak{0}}
\newcommand{\signu}{s(u)}
\newcommand{\dtop}{D^\text{top}}
\newcommand{\lacan}{\lambda_{\text{can}}}
\newcommand{\lamod}{\lambda_{\text{mod}}}

\date{\today}

\begin{document}

\begin{abstract}
We introduce a class of Weinstein domains which are sublevel sets of flexible Weinstein manifolds but are not themselves flexible. These manifolds exhibit rather subtle behavior with respect to both holomorphic curve invariants and symplectic flexibility.
We construct a large class of examples and prove that every flexible Weinstein manifold can be Weinstein homotoped to have a nonflexible sublevel set.
\end{abstract}

\title{Subflexible symplectic manifolds}
\author{Emmy Murphy}

\thanks{The first author was partially supported by NSF grant DMS-1510305.}
\author{Kyler Siegel}
\thanks{The second author was partially supported by NSF grant DGE-114747.}

\maketitle

\tableofcontents

\section{Introduction}

It has been known since the work of Gromov \cite{gromov2013partial} that subcritical isotropic submanifolds of symplectic and contact manifolds satisfy an h-principle, meaning they belong to the realm of algebraic topology. In the intervening time, a rich theory of holomorphic curve invariants has shown that Lagrangians and Legendrians are generally quite rigid geometric objects. However, the last few years have seen significant progress in the flexible side of symplectic topology. The first author's discovery of {\em loose Legendrians} \cite{murphy2012loose} shows that, at least for isotropics in high dimensional contact manifolds, symplectic flexibility extends well into the critical dimension. Cieliebak--Eliashberg subsequently used loose Legendrians to define {\em flexible Weinstein manifolds}, extending this flexibility to the theory of symplectic handlebodies. In a slightly different direction, it was shown in \cite{eliashberg2013lagrangian} that Lagrangian embeddings with a loose Legendrian negative end also satisfy an h-principle. Applications of this have included Lagrangian immersions with surprisingly few self-intersection points \cite{ekholm2013constructing}, an h-principle for symplectic embeddings of flexible Weinstein manifolds \cite{eliashberg2013lagrangian} (see also Theorem \ref{thm:subflex embed} below), and a complete classification of the smooth topology of polynomially and rationally convex domains in high dimensional affine space \cite{cieliebak2013topology}.

On the other hand, certain questions about the precise nature of these flexible objects have been thus far unclear. 
One question raised by Cieliebak--Eliashberg (see Remark 11.30(3) in \cite{cieliebak2012stein}) is whether the notion of flexibility for Weinstein manifolds is invariant under Weinstein homotopies, or if it somehow depends on how we chop up the manifold into elementary pieces.
A closely related question is whether a subdomain of a flexible Weinstein manifold is necessarily flexible.
More specifically, is every sublevel set of a flexible Weinstein structure on $\C^n$ necessarily flexible?
By the work of Cieliebak--Eliashberg (see the proof of Theorem \ref{thm:pc domains} below), this is equivalent to asking whether every polynomially convex domain in $\C^n$ is necessarily flexible, and the affirmative answer was conjectured in \cite{cieliebak2013topology}.

In this paper, our main result is the following:
\begin{thm}\label{thm:mainthmintro}
Every flexible Weinstein manifold has, after a Weinstein homotopy, a nonflexible sublevel set.
\end{thm}
Such sublevel sets, which we call {\em subflexible}, lie close to the interface between flexibility and rigidity.
We construct a large class of examples,
including many which are sublevel sets of the standard Weinstein structure on $\C^n$ up to Weinstein homotopy. In particular this gives a negative answer to all of the above questions and disproves the conjecture of Cieliebak--Eliashberg.

Our starting observation is that the exotic $6$-dimensional Weinstein manifold $X$ first defined by Maydanskiy in \cite{maydanskiy2009} becomes flexible after attaching an additional critical Weinstein handle (see $\S\ref{subsec:Maydanskiy}$). 
This gives another viewpoint on Maydanskiy's result that $X$ has trivial wrapped Fukaya category (see $\S\ref{sec: holomorphic curve invariants}$ below, along with \cite[Proposition 2.3]{abouzaid2010altering} for the connection with wrapping).
On the other hand, Harris \cite{harris2012distinguishing} observed that $X$ contains a Lagrangian $S^3$ after arbitrarily small non-exact deformations of the symplectic form.
Applying Eliashberg--Murphy's h-principle for embeddings of flexible Weinstein manifolds, we can summarize the discussion so far as:
\begin{thm}\label{thm:X in C3}
In the standard symplectic $\C^3$, there is a Liouville subdomain $X$ and a $C^\infty$-small closed form $\eta \in \Omega^2X$, so that the symplectic manifold $(X, \omega_\std|_X + \eta)$ contains a Lagrangian $S^3$.
\end{thm}
In fact, Harris' observation suggests that $X$ might be nonflexible, although deducing this directly from the above theorem appears tricky. The reason is that after deformation $X$ becomes non-exact at infinity, and presently there is not a robust theory of holomorphic curve invariants for such manifolds. 

Inspired by the above example and the ``homologous recombination" construction of Abouzaid--Seidel \cite{abouzaid2010altering}, we introduce a general ``subflexibilization" construction for Weinstein manifolds.
Subflexibilization inputs a Weinstein manifold $W$ and outputs a Weinstein manifold $\sbf(W)$,
and applying the construction to $T^*S^3$ gives Maydanskiy's manifold.
Topologically this procedure has the effect of adding some subcritical Weinstein handles. Symplectically it renders $W$ subflexible, in particular killing its symplectic cohomology, and yet the symplectic geometry of $W$ is not completely forgotten as we will explain.
In fact, Harris' observation can be understood as a manifestation of Seidel's result \cite{seidel1997floer} that squares of two-dimensional Dehn twists are {\em fragile},
i.e. are not symplectically isotopic to the identity, but become so 
after small non-exact deformations of the symplectic form.
Correspondingly, for $6$-dimensional $W$ we identify the symplectomorphism type of $\sbf(W)$ after a small non-exact deformation as 
 the boundary connect sum of $W$ with some simple standard (non-exact) symplectic manifold (see Theorem \ref{deformationthm}).

Detecting nonflexibility among subflexible manifolds is a rather subtle problem. In general the basic tool for detecting nonflexibility of a Weinstein manifold is symplectic cohomology.
Indeed, the work of Bourgeois--Ekholm--Eliashberg \cite{bourgeois2012effect} implies that 
the symplectic cohomology of a flexible Weinstein manifold must vanish. We give an alternative proof of this in $\S \ref{sec: holomorphic curve invariants}$ which may be of independent interest. 
However, a standard argument involving the Viterbo transfer map implies that subflexible Weinstein manifolds must have vanishing symplectic cohomology as well, so some other invariant is needed.
For this we turn to a twisted variation of symplectic cohomology.
Twisted symplectic cohomology also vanishes for {\em flexible} Weinstein manifolds, but, perhaps surprisingly at first glance, {\em subflexible} manifolds can have nontrivial twisted symplectic cohomology. 
In fact, the results in \cite{Siegel} of the second author imply that, for $6$-dimensional $W$, twisted symplectic cohomology of $\sbf(W)$ coincides with standard symplectic cohomology of $W$. 
Appealing to this computation, we deduce that {\em$\sbf(W)$ is subflexible, yet nonflexible whenever $W$ has nonvanishing symplectic cohomology}.

\begin{remark}
In fact, \cite{Siegel} also computes {\em bulked deformed symplectic cohomology} for an analogous class of examples, and this can be used to distinguish exotic examples for which ordinary or twisted symplectic cohomology cannot.
However, for simplicity we focus on twisted symplectic cohomology in this paper, especially since this suffices to prove Theorem \ref{thm:mainthmintro}.
\end{remark}

Applying the construction with a bit more care and incorporating other recent constructions of exotic Weinstein manifolds, we prove the following.
Recall that an {\em almost symplectomorphism} is a diffeomorphism which furthermore preserves the homotopy class of the symplectic form as a nondegenerate two-form. 
\begin{thm} \label{thm:advertising}
Let $X$ be any Weinstein domain with $\dim X \geq 6$ and $c_1(X) = 0$.
Then there is a Weinstein domain $X'$ such that
\begin{itemize}
\item
$X'$ is almost symplectomorphic to the boundary connect sum of $X$ with some standard subcritical Weinstein domain
\item $X'$ is a sublevel set of a flexible Weinstein domain almost symplectomorphic to $X$
\item $X'$ has nonvanishing twisted symplectic cohomology.
\end{itemize}
In particular, $X'$ is subflexible but not flexible.
\end{thm}
Taking $X$ to be the standard Weinstein ball, we get a nonflexible sublevel set $X'$ of a flexible Weinstein domain which is almost symplectomorphic to $X$, and hence Weinstein deformation equivalent to $X$ by the h principle for flexible Weinstein domains (see \S\ref{subsec:looseness and flexibility}). Since every Weinstein manifold can be homotoped so that it contains the ball as a small sublevel set, this establishes Theorem \ref{thm:mainthmintro}.

Incorporating the techniques of Cieliebak--Eliashberg \cite{cieliebak2013topology},
we also prove:
\begin{thm}\label{thm:pc domains}
If $X$ admits a smooth codimension $0$ embedding into $\C^n$ and $H_n(X;\Z) = 0$ and $H_{n-1}(X;\Z)$ is torsion-free, then $X'$ is a sublevel set of $\C^n$, equipped with the standard Weinstein structure up to deformation. In particular, $X'$ is Weinstein deformation equivalent to a polynomially convex
domain in $\C^n$.
\end{thm}

This paper is organized as follows. In $\S \ref{sec:background}$ we review the relevant 
background on Weinstein and Lefschetz structures and especially the interplay between the two. 
In $\S \ref{sec: holomorphic curve invariants}$ we discuss symplectic cohomology and its twisted cousin and prove that these invariants vanish for flexible Weinstein structures.
In $\S \ref{sec:subflexibilization}$ we introduce the subflexibilization construction and establish its main properties, using the groundwork laid in $\S \ref{sec:background}$. Finally, in $\S \ref{sec:applications}$ we combine all of our results to produce some exotic subflexible manifolds as in Theorem \ref{thm:advertising}.
\section*{Acknowledgements}
{
We wish to thank Yasha Eliashberg for inspiring the ideas in this paper and countless enlightening conversations. We also thank Oleg Lazarev for a careful reading and various helpful comments on an earlier draft of this paper, Mark McLean for discussions concerning vanishing symplectic cohomology of flexible Weinstein manifolds, and Aleksander Doan for sharing his work on symplectic cohomology 
in the non-convex setting and many stimulating discussions thereof.
}
      
\section{Background}\label{sec:background}

\subsection{The Geometry of Weinstein structures}

We begin with a quick review of open symplectic structures, refering the reader to \cite{cieliebak2012stein} for more details.
Recall that a \emph{Liouville domain} is a pair $(W^{2n+2},\la)$, where 
\begin{itemize}
\item
$W$ is a compact $(2n+2)$-dimensional manifold with boundary
\item 
$\la$ is a $1$-form on $W$ such that $\omega := d\la$ is symplectic
\item 
the Liouville vector field $Z_\la$, defined by $Z_\la \ip d\la = \la$, is outwardly transverse to $\bdy W$.
\end{itemize}
There is also a weaker notion of a {\em compact symplectic manifold with convex boundary} $(W,\omega)$, in which $\omega$ has a primitive $\la$ defined only near $\bdy W$ such that $Z_\la$ is outwardly transverse to $\bdy W$.

A \emph{Weinstein domain} is a triple $(W, \lambda, \p)$, where
\begin{itemize}
\item
$(W,\la)$ is a Liouville domain
\item 
$\phi: W \rightarrow \R$ is a Morse function with maximal level set $\bdy W$
\item 
$Z_\la$ is gradient-like for $\phi$.
\end{itemize}
A Weinstein domain has a \emph{completion} $\wh W = W \cup \left([0,\infty) \x \dd W\right)$.
Let $r$ denote the unique flow coordinate for $Z_\la$ on $\Op(\bdy W)$ satisfying 
and $\calL_{Z_\la}(r) \equiv 1$ and $r|_{\bdy W} \equiv 0$.
 Assuming $\p \equiv r$ on $\Op(\bdy W)$, we extend $\lambda$ to $[0,\infty) \x \dd W$ as $\wh \lambda = e^r\lambda|_{\dd W}$ and we extend $\p$ as $\wh \p(r, x) = r$. Here $\wh W$ is an open manifold without boundary, $\wh \p: \wh W \to [0,\infty)$ is a proper Morse function, and the flow of the vector field $Z_{\wh\lambda}$ is complete. The completion $(\wh W, \wh\lambda, \wh \p)$ is called a \emph{Weinstein manifold}. In this paper we require Weinstein manifolds to be \emph{finite type}, meaning the proper Morse function has only finitely many critical points. Any finite type Weinstein manifold is the completion of a Weinstein domain.

For two Weinstein domains $(W,\la_1,\phi_1),(W,\la_2,\phi_2)$ with the same underlying smooth manifold,
the natural equivalence relation is {\em Weinstein homotopy}, i.e. a $1$-parameter family of Weinstein structures $(W,\la_t,\phi_t)$ connecting them, $t \in [1,2]$, where $\phi_t$ is additionally allowed to have standard birth-death singularities at finitely many times.
For Weinstein manifolds, the definition is similar, but with the added stipulation that the union of the critical points of $\phi_t$ for all $t$ be contained in some compact subset (this is to prevent 
critical points from disappearing off to infinity).
Two Weinstein structures\footnote{Here {\em structure} means either domain or manifold.} on a priori different smooth manifolds are \emph{deformation equivalent} if, after pulling back one structure by a diffeomorphism, the two resulting structures on the same smooth manifold are Weinstein homotopic.
For Weinstein manifolds, Weinstein homotopy (and hence deformation equivalence) implies exact symplectomorphism.
Note however that two homotopic Weinstein {\em domains} need not be symplectomorphic, since for example their volumes or symplectic capacities could be different. For this reason, regarding questions of symplectomorphisms it is more natural to work with completions.  

The definition of a Weinstein domain implies that $\lambda$ is a contact form when restricted to $Y^c := \p^{-1}(c)$ for any regular value $c$. Furthermore, the descending manifold $D^k_p$ of any critical point $p \in W$ satisfies $\lambda|_{D^k_p} = 0$. Therefore $D^k_p$ is isotropic in the symplectic sense for $d\lambda$, and $\Lambda^c_p := D^k_p \cap Y^c$ is isotropic in the contact sense for $\ker \lambda|_{Y^c}$. In particular $k := \op{ind}(p) \leq n+1$. 

If $c \in \R$ is a critical value of $\p$ with a unique critical point $p$, then the Weinstein homotopy type of $W^{c+\e} = \{\phi < c + \e\}$ (for $\e>0$ sufficiently small) is determined by $W^{c-\e}$, together with the isotopy type\footnote{$\Lambda^{c-\e}_p$ is canonically parametrized as $\bdy D^{k+1}_p$, and its isotopy type through \emph{parametrized} Legendrians affects the symplectomorphism type of $W^{c+\e}$. We will often suppress this parametrization from the notation.}  of $\Lambda_p^{c-\e} \sse Y^{c-\e}$ and a framing of the symplectic normal bundle of $\La_p^{c-\e}$ (which is necessarily trivial). This is also constructive: given a Weinstein domain $W$ and a (parametrized) isotropic sphere $\La^k \sse  \dd W$, together with a framing of the symplectic normal bundle of $\La$ (assumed to be trivial), we can construct a new Weinstein domain with one additional critical point, of index $k+1$, whose descending manifold intersects $\bdy W$ along precisely $\Lambda$. 
This procedure, called {\em attaching a Weinstein $(k+1)$-handle along $\La$}, depends only on the framed, parametrized isotopy class (through isotropics) of $\Lambda$, up to Weinstein homotopy (see \cite{weinstein1991contact, cieliebak2012stein} for more details).

In particular, for a Legendrian sphere $\La^n \subset \bdy W$, we can attach an $(n+1)$-handle $H \cong \D^{n+1}\times \D^{n+1}$ along $\Lambda$ and extend the Weinstein structure over $H$. The resulting Morse function agrees with $\phi$ on $W$ and has one additional critical point $p_H$ of index $n+1$ in $H$. 
The {\em core} (resp. {\em cocore}) of $H$ are the Lagrangian disks consisting of all points in $H$ which limit to $p_H$ under the positive (resp. negative) Liouville flow.
The boundary of the core (resp. cocore) of $H$ is a Legendrian sphere in $\bdy W$ (resp. $\bdy(W \cup H$)), which we refer to as the \emph{attaching sphere} (resp. \emph{belt sphere}) of $H$.
We denote the attaching sphere and belt sphere of $H$ by $\as(H)$ and $\bs(H)$ respectively.

In general, if a critical point $p$ of $\phi: W^{2n+2} \rightarrow \R$ has $\op{ind}(p) = n+1$, we say that $p$ is \emph{critical}; otherwise it is \emph{subcritical}. 
A Weinstein manifold is called \emph{explicitly subcritical} if all of its critical points are subcritical. Note that explicit subcriticality is not invariant under Weinstein homotopy, since one can easily perform a Weinstein homotopy which creates two canceling critical points, one of index $n$ and one of index $n+1$ (see \cite[\S 12.6]{cieliebak2012stein}).
We will call a Weinstein manifold \emph{subcritical} if it can be made explicitly subcritical by a Weinstein homotopy.

\subsection{Looseness and flexibility}\label{subsec:looseness and flexibility}

\emph{Loose Legendrians} were defined in \cite{murphy2012loose}, where it was shown that they satisfy an h-principle: two loose Legendrians are Legendrian isotopic if and only if they are \emph{formally isotopic}. Intuitively, two Legendrians are formally isotopic if and only if they are smoothly isotopic in such a way that their normal bundle framings are canonically homotopic.
A connected Legendrian of dimension $n > 1$ is loose if it admits a \emph{loose chart}, defined as follows.
Let $\gamma_a$ denote the Legendrian arc in $(B^3_\std,\bdy B^3_\std)$, defined up to Legendrian isotopy, with the properties depicted in the left side of Figure \ref{stabilizedarcandzigzag}. Namely, it has a single Reeb chord, of action $a$, and its front projection has a single transverse self-intersection and a single cusp.
Set
\begin{align*}
V_\rho^{2n-2} &:= \{(q,p) \in T^*\R^{n-1}\;:\; |q| \leq \rho,\; |p| \leq \rho\}\\
Z_\rho^{n-1} &:= \{(q,p) \in V_\rho^{2n-2}\;:\; p = 0\}.
\end{align*}
Here $T^*\R^{n-1}$ is equipped with the Liouville form $-p_1dq_1 - ... - p_{n-1}dq_{n-1}$,
and we equip the product $B^3_\std \times V_\rho$ with the contact form
$z - ydx - p_1dq_1 - ... - p_{n-1}dq_{n-1}$.
A loose chart for a Legendrian $\La \subset N$ is a contact embedding of pairs
\begin{align*}
(B^3_\std \times V_\rho, \gamma_a \times Z_\rho) \hookrightarrow (N,\La)
\end{align*}
such that $\frac{a}{\rho^2} < 2$.
More generally, we say a Legendrian link
is loose if each component admits a loose chart in the complement of the other components. 

\begin{figure}
 \centering
 \includegraphics[scale=.8]{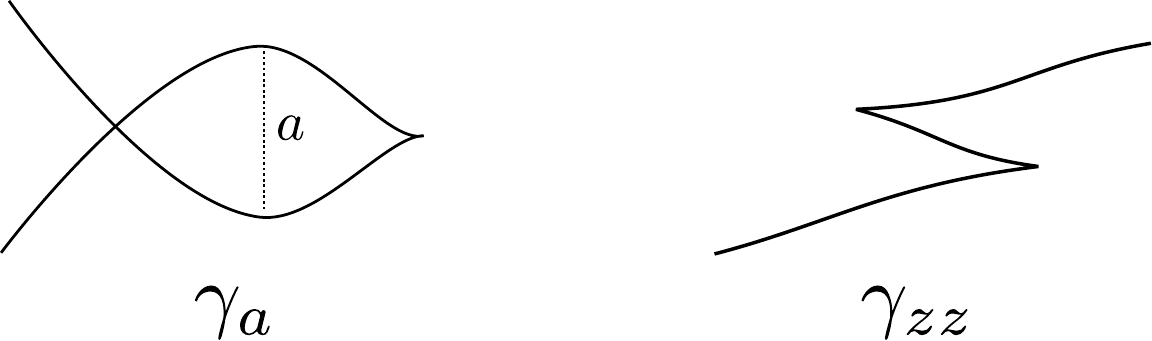}
 \caption{Left: the front picture of the Legendrian arc $\gamma_a$ with a chord of action $a$. Right: the front picture of a once stabilized Legendrian arc.}
 \label{stabilizedarcandzigzag}
\end{figure} 

Since loose Legendrians satisfy an h-principle, one might also expect an h-principle for Weinstein manifolds built by iterative Weinstein handle attachments along loose Legendrian spheres. 
This indeed turns out to be the case and is explored in depth in \cite{cieliebak2012stein}.

\begin{definition}\label{def:flexible}
For $n>1$, we say a Weinstein structure $(W^{2n+2},\la,\phi)$ is {\em explicitly flexible} if there exist regular values $c_0,c_1,...,c_N$ of $\phi$
with $c_0 < \min(\phi) < c_1 < ... < c_N$ such that
\begin{itemize}
\item all critical points of $\phi$ are contained in $\{ \phi < c_N\}$
\item there are no gradient trajectories of $Z_\la$ joining two critical points in $\{ c_i < \phi < c_{i+1}\}$ for $i = 0,...,N-1$ (i.e. $\{c_i \leq \phi \leq c_{i+1}\}$ is an {\em elementary cobordism} in the language of \cite{cieliebak2012stein})
\item the attaching spheres of all index $n+1$ critical points in $\{ c_i  < \phi < c_{i+1}\}$ form a loose Legendrian link in $\phi^{-1}(c_i)$ for $i = 1,...,N-1$.
\end{itemize}
A Weinstein structure is called \emph{flexible} if it is Weinstein homotopic to an explicitly flexible structure.
\end{definition}
\begin{remark}\label{remark:explicitflexcrit}
\hspace{1cm}
\begin{enumerate}
\item
Note that any (explicitly) subcritical Weinstein domain is (explicitly) flexible.
\item 
In a slight change of terminology, our definition of {\em explicitly flexible} coincides with 
the definition of {\em flexible} given in \cite{cieliebak2012stein}. 
\item 
As explained in \cite[Remark 11.30(3)]{cieliebak2012stein}, if $c_0 < c_1 < ... < c_N$ is a partition as in Definition \ref{def:flexible},
so is any finer partition of $(W,\la,\phi)$ into elementary cobordisms. In particular, if the critical points of $\phi$ have pairwise distinct critical values
then $(W,\la,\phi)$ is explicitly flexible if and only if the attaching Legendrian of each index $n+1$ critical point $p$ is loose in $\phi^{-1}(\phi(p)-\e)$ for $\e>0$ sufficiently small.
If some of the critical points of $\phi$ share the same critical value $q$, a similar criterion holds if we consider all the corresponding attaching spheres as a Legendrian link in $\phi^{-1}(q-\e)$.
\end{enumerate}
\end{remark}

To justify the name, Cieliebak--Eliashberg prove the following flexibility results (see \cite{cieliebak2012stein} for stronger and more precise statements):

\begin{thm}\label{thm:CE flex}
\hspace{1cm}
\begin{enumerate}
\item
Given an explicitly flexible Weinstein structure $(W^{2n+2},\la,\phi)$ and any Morse function $\wt{\phi}: W \rightarrow \R$ without critical points of index greater than $n+1$, there is a Weinstein homotopy starting at $(W,\la,\phi)$ and ending at a Weinstein structure whose Morse function is $\wt{\phi}$.
\item 
Two explicitly flexible Weinstein structures are Weinstein homotopic if and only if their symplectic forms are homotopic as non-degenerate two-forms.
\item
Any Weinstein structure $(W,\la,\phi)$ is {\em almost symplectomorphic} to a flexible Weinstein structure, i.e. there is a diffeomorphism respecting the homotopy classes of the symplectic forms as nondegenerate two-forms.
We denote this flexible version by \emph{$\flex(W,\la,\phi)$}. Note that it is well-defined up to Weinstein deformation equivalence.
\end{enumerate}
\end{thm}

One of the main goals of this paper is to show that explicit flexibility is \emph{not} preserved under Weinstein homotopies. 
To see that this is conceivably possible, imagine a Weinstein domain $(W,\la,\phi)$ with precisely two critical points of critical index, say $p_1$ and $p_2$,
such that $\phi(p_1) = 1$ and $\phi(p_2) = 2$. Assume there are no gradient trajectories between $p_1$ and $p_2$, and
suppose that $(W,\la,\phi)$ is explicitly flexible, which means that $\as(p_1) \subset \phi^{-1}(1/2)$ and $\as(p_2) \subset \phi^{-1}(3/2)$ are loose Legendrians.
Now suppose we were to homotope $(W,\la,\phi)$ by lowering the value of $p_2$ until $p_1$ and $p_2$ lie on the same level set of $\phi$.
The result would be explicitly flexible if and only if $\as(p_1) \cup \as(p_2)$ is a loose link in $\phi^{-1}(1/2)$.
Since we did not assume $\as(p_2) \subset \phi^{-1}(3/2)$ is loose in the complement of $\bs(p_1)$, there is no obvious reason why $\as(p_2) \subset \phi^{-1}(1/2)$ should be 
loose in the complement of $\as(p_1)$.

Note, however, that Theorem \ref{thm:CE flex} above is tautologically also true for the more general class of flexible Weinstein manifolds.
\begin{definition} \label{def:subflexible}
A Weinstein domain $(W, \lambda, \p)$ is called \emph{subflexible} if it is Weinstein deformation equivalent
 to a sublevel set of a flexible Weinstein manifold. 
\end{definition}
\begin{remark}
A sublevel set of a Weinstein domain is a special case of a \emph{Liouville embedding}, i.e. a smooth codimension zero embedding $i: (W_0,\la_0) \hookrightarrow (W,\la)$ of a Liouville domain into a Liouville manifold such that $i^* \la - e^\rho\la_0$ is exact for some $\rho \in \R$. For Liouville embeddings one can define the Viterbo transfer map on symplectic cohomology \cite{viterbo1999functors}.
\end{remark}

Flexible manifolds are also subflexible by definition, and at first glance it might appear
that these are the only examples.
For instance, any subflexible Weinstein domain has trivial symplectic cohomology (see Proposition \ref{thm:vanishing SH}).
On the other hand, as we explain in $\S \ref{sec: holomorphic curve invariants}$, the proof of Proposition \ref{thm:vanishing SH} fails for a twisted version of symplectic cohomology (even though a version of the transfer map still holds in this setting!). 
In fact, twisted symplectic cohomology is strong enough to detect nonflexibility of the examples we construct in $\S \ref{sec:subflexibilization}$.

\subsection{Lefschetz structures}\label{sec:Lefschetz structures}
Roughly speaking, a smooth map $\pi: E^{2n} \rightarrow B^2$ is a {\em smooth Lefschetz fibration} if $B^2$ is a compact surface with boundary, $E^{2n}$ is a compact manifold with 
corners\footnote{From now on we will not explicitly mention the corners and will assume the corners have been smoothed when needed.}, and $\pi$ is a submersion except at finitely many singular points, near which it is modeled on the holomorphic map 
\begin{align*}
\C^n \rightarrow \C,\;\;\;\;\; (z_1,...,z_n) \mapsto z_1^2 + ... + z_n^2.
\end{align*}
If $E$ has additional structure, we would like $\pi$ to be compatible with this structure in some sense.
Roughly speaking, a {\em symplectic Lefschetz fibration} is a Lefschetz fibration such that the total space is equipped with a symplectic form which restricts to a symplectic form on each fiber, with some additional technical conditions to ensure parallel transport is well-behaved.
We now give an actual definition with the caveat that the precise details are not crucial for our purposes and there seems to be no universally agreed upon definition in the literature.
\begin{definition}
A {\em symplectic Lefschetz fibration} is a smooth map $\pi: W \rightarrow \D^2$, with a $(W,\omega)$ a compact symplectic manifold, such that
\begin{itemize}
\item
{\em Lefschetz type singularities:} $\pi$ is a submersion whose fibers are smooth manifolds with boundary, except at finitely many critical points. The critical points have distinct critical values lying in $\Int \D^2$, and near each critical point $\pi$ is modeled on $\C^n \rightarrow \C$, $(z_1,...,z_n) \mapsto z_1^2 + ... + z_n^2$, with $\omega$ identified with the standard K\"ahler form on $\C^n$. 
\item 
{\em Compatibility with $\omega$:} For any regular value $p \in \D^2$, $(\pi^{-1}(p),\omega|_{\pi^{-1}(p)})$ is a compact symplectic manifold with convex boundary
\item
{\em Triviality near the vertical boundary:} On a neighborhood of $\bdy_vW := \pi^{-1}(\bdy \D^2)$, $\pi$ is equivalent (for some $\e > 0$) to the map 
\begin{align*}
\pi \times \op{Id}: \bdy_v W \times (1-\e,1] \rightarrow S^1 \times (1-\e,1] \subset \D^2
\end{align*}
with $\omega$ identified with $\omega_v + \pi^*\omega_b$,
where $\omega_v$ is the pullback of $\omega|_{\bdy_vW}$ under the projection
$\bdy_vW \times (1-\e,1] \rightarrow \bdy_v W$ and $\omega_b$ is some symplectic two-form on $\D^2$.
\item 
{\em Triviality near the horizontal boundary:} On a neighborhood of $\bdy_hW := \bdy W \setminus \Int\bdy_vW$, $\pi$ is equivalent to the projection $\Op(\bdy M) \times \D^2 \rightarrow \D^2$, with $\omega$ identified with a split symplectic form, where $M := \pi^{-1}(1)$ is the fiber.
\end{itemize}
\end{definition}
The key feature of symplectic Lefschetz fibrations is that the symplectic orthogonals to the vertical tangent spaces define a symplectic connection, meaning any path $\gamma: [0,1] \rightarrow \D^2$ which avoids the critical values of $\pi$ induces a parallel transport symplectomorphism $\pi^{-1}(\gamma(0)) \cong \pi^{-1}(\gamma(1))$.
The symplectic Picard--Lefschetz theorem identifies the holonomy around a critical value of $\pi$ with a symplectic Dehn twist along the corresponding vanishing cycle (see below).
The notion of a {\em Liouville Lefschetz fibration} is similar, with $\omega = d\la$ a Liouville structure on $W$ and $(\pi^{-1}(p),\la|_{\pi^{-1}(p)})$ a Liouville domain for each regular value $p \in \D^2$. We refer the reader to \cite{seidelbook} for a comprehensive treatment.

Let $(W,\omega)$ be a symplectic manifold equipped with a symplectic Lefschetz fibration $\pi: W \rightarrow \D^2$, and let $z_1,...,z_k \in \D^2$ denote the critical values of $\pi$.
A {\em basis of vanishing paths} is a collection of embedded paths $\gamma_1,...,\gamma_k: [0,1] \hookrightarrow \D^2$ such that for $i = 1,...,k$, $\gamma_i(0) = 1 \in \bdy\D^2$, $\gamma_i(1) = z_i$, and $\gamma_i|_{(0,1)}$ is an embedding in $\Int \D^2 \setminus \{z_1,...,z_k\}$.
We associate to a basis of vanishing paths the {\em vanishing cycles} $V_1,...,V_k \subset \pi^{-1}(1)$,  where $V_i$ is defined as the set of all points in $M := \pi^{-1}(1)$ which are parallel transported along $\gamma_i$ to the critical point in $\pi^{-1}(z_i)$.
The vanishing cycles $V_1,...,V_k$ are embedded parametrized\footnote{Just as in the case of Weinstein manifolds and Legendrian attaching maps, vanishing cycles have a canonical parametrization (modulo isotopies) since they are the boundaries of the \emph{Lagrangian thimbles}. Similar to the Weinstein case, this data is necessary to recover the symplectic topology of the total space from the data of the fiber and the vanishing cycles.}
Lagrangian spheres in $(M,\omega|_M)$.

Now suppose that $(V_1,...,V_k)$ are the vanishing cycles associated to a certain basis of vanishing paths, and $(V_1',...,V_k')$ are those associated to some other basis of vanishing paths.
It follows from the symplectic Picard--Lefschetz theorem
that $(V_1,...,V_k)$ and $(V_1',...,V_k')$ differ by a sequence of {\em Hurwitz moves}. 
A {\em Hurwitz move} is one of the following two operations:
\begin{align*}
(V_1,...,V_k) \rightsquigarrow (\tau_{V_1}V_2,V_1,V_3,...,V_k)\\
(V_1,...,V_k) \rightsquigarrow (V_2,\tau_{V_2}^{-1}V_1,V_3,...,V_k).
\end{align*}
(Here and throughout the paper, we denote by $\tau_V$ the symplectic Dehn twist around the Lagrangian sphere $V$.) Note that the two moves are inverses of each other, and as usual the vanishing cycles are only cyclically ordered, so the moves can be applied on any two consecutive vanishing cycles.

In general, we call $(M,\omega;V_1,...,V_k)$ a {\em symplectic Lefschetz datum} if $(M,\omega)$ is a compact symplectic manifold with convex boundary and $V_1,...,V_k \subset M$ is a cyclically ordered collection of embedded parametrized Lagrangian spheres. 
Given a symplectic Lefschetz datum $(M,\omega;V_1,...,V_k)$ and a basis of vanishing paths,
one can construct a compact symplectic manifold $\fS(M,\omega;V_1,...,V_k)$, equipped with a symplectic Lefschetz fibration 
\begin{align*}
\pi_{(M,\omega;V_1,...,V_k)}:\fS(M,\omega;V_1,...,V_k) \rightarrow \D^2.
\end{align*}
Moreover, the fiber is symplectomorphic to $M$ and the vanishing cycles are identified with $V_1,...,V_k$.
Similarly, a Liouville Lefschetz datum is $(M,\la;V_1,...,V_k)$ with $(M,\la)$ a Liouville domain,
and one can construct a Liouville Lefschetz fibration
\begin{align*}
\pi_{(M,\la;V_1,...,V_k)}: \fL(M,\la;V_1,...,V_k) \rightarrow \D^2
\end{align*}
which recovers the initial datum $(M,\la;V_1,...,V_k)$.

The construction of $\pi_{(M,\omega;V_1,...,V_k)}:\fS(M,\omega;V_1,...,V_k) \rightarrow \D^2$ proceeds as follows. 
Of course if $k=0$ we can simply take $M \times \D^2$ with the split symplectic form $\omega + \omega_\std$.
Now suppose $k=1$. For any $s > 0$ there is a ``model" symplectic Lefschetz fibration $\pi: E_s \rightarrow \D^2$ whose fiber is symplectomorphic to $D^*_s S^n$, the disk cotangent bundle of $S^n$ of radius $s$ with respect to some Riemannian metric (see \cite[Example 15.4]{seidelbook}). 
Near $\bdy_hE_s$, $\pi$ is equivalent to the projection $\Op(\bdy D^*_s S^n) \times \D^2 \rightarrow \D^2$. Here $\Op(\bdy D^*_s S^n) \times \D^2$ is equipped with the symplectic form $\omega_\std + \omega_b$, where $\omega_\std$ is the canonical symplectic form on $D^*_s S^n$ and $\omega_b$ is some symplectic form on $\D^2$. 
Now consider $M \times \D^2$ with the split symplectic form $\omega + \omega_b$.
By the Weinstein neighborhood theorem, a neighborhood $U$ of $V_1 \times \D^2$ is symplectomorphic to $(D^*_sS^n \times \D^2,\omega_\std + \omega_b)$ for some $s > 0$. We can therefore remove $U$ from $M \times \D^2$ and symplectically glue in $E_s$, and this gives the desired symplectic Lefschetz fibration.
More generally, if $k \geq 2$ we perform the above construction for each vanishing cycle individually and then glue together the resulting symplectic Lefschetz fibrations along a fiber, reidentifying the base with $\D^2$. 
The Liouville Lefschetz fibration ${\pi_{(M,\la;V_1,...,V_k)}:\fL(M,\la;V_1,...,V_k) \rightarrow \D^2}$ can be constructed in a similar fashion (see \cite[\S 16e]{seidelbook} for more details).

There is also a notion of {\em completion} for a Lefschetz fibration over $\D^2$, resulting in a Lefschetz fibration over $\C$.  If $\pi: W \rightarrow \D^2$ is a symplectic Lefschetz fibration, the completion $\wh{\pi}: \wh{W} \rightarrow \C$ is obtained as follows. First, by triviality near the horizontal boundary we can symplectically glue $(\wh{M} \setminus M) \times \D^2$ to $W$, where $\wh{M}$ denotes the completion of $M = \pi^{-1}(1)$ as a compact symplectic manifold with convex boundary. Denote the resulting symplectic manifold by $(W_1,\omega_1)$, and note that $\pi$ has a natural extension $\pi_1: W_1 \rightarrow \D^2$. 
Now by triviality near the vertical boundary we can identify the restriction of $\pi_1$ to $\Op\left(\pi_1^{-1}(\bdy \D^2)\right)$ with the map
\begin{align*}
\pi \times \op{Id}: \bdy_vW_1\times (1-\e,1] \rightarrow \D^2.
\end{align*}
Here $\bdy_vW_1\times (1-\e,1] $ is equipped with the symplectic form $\wh{\omega_v} + \pi^*\omega_b$ where $\wh{\omega_v} := \omega_1|_{\bdy_v W_1}$.
Let $(\C,\wh{\omega_b})$ denote the completion of $(\D^2,\omega_b)$.
Finally, we complete the horizontal boundary of $W_1$ by gluing in $\bdy_v W_1 \times [1,\infty)$,
equipped with the symplectic form $\wh{\omega_v} + \wh{\omega_b}$.
One can similarly complete a Liouville Lefschetz fibration, and in this case the completion as a Lefschetz fibration agrees with the completion as a Liouville domain, up to Liouville deformation equivalence. We denote the completions of $\fS$ and $\fL$ by $\fShat$ and $\fLhat$ respectively. 

We compile a few basic facts that will be needed later:
\begin{itemize}
\item
The symplectomorphism type of $\fShat(M,\omega;V_1,...,V_k)$ is well-defined and invariant under symplectomorphisms of $M$ and Hamiltonian isotopies and Hurwitz moves of the vanishing cycles.
\item 
$\fL(M,\la;V_1,...,V_k)$ is well-defined up to Liouville deformation equivalence and invariant under Liouville deformation equivalences of $(M,\la)$ and Hamiltonian isotopies and Hurwitz moves of the vanishing cycles.
\item 
$\fLhat(M,\la;V_1,...,V_k)$ is symplectomorphic to $\fShat(M,d\la;V_1,...,V_k)$.
\item
Let $(M,\omega;V_1,...,V_k)$ be a symplectic Lefschetz datum and let $\Omega$ be a closed two-form on $M$ with support disjoint from $V_1 \cup ... \cup V_k$.
Then there is a closed two-form $\wt{\Omega}$ on $\fShat(M,\omega;V_1,...,V_k)$ such that 
$\fShat(M,\omega + s\Omega;V_1,...,V_k)$ is symplectomorphic to the result of adding $s\wt{\Omega}$ to the symplectic form of $\fShat(M,\omega;V_1,...,V_k)$
(provided $(M,\omega + s\Omega)$ is symplectic with convex boundary).
\end{itemize}

As explained in \cite[\S 8.2]{bourgeois2012effect}, the Liouville domain $\fL(M,\la;V_1,...,V_k)$ can also be constructed (up to Liouville deformation equivalence) by attaching Weinstein handles to $M \times \D^2$ along Legendrian lifts of $V_1,...,V_k$. If $(M,\la,\phi)$ is a Weinstein domain, this can be slightly refined to construct a Weinstein domain $\fW(M,\la,\phi;V_1,...,V_k)$, as we now explain.
Let $(M,\la,\phi)$ be a Weinstein domain, and let $(\D^2,\la_\std,\phi_\std)$ be the standard Weinstein structure on the unit disk with $\la_\std = \frac{1}{2}xdy - \frac{1}{2}ydx$ and $\phi_\std(x,y) = x^2 + y^2$. We would like to view the product $(M,\la,\phi) \times (\D^2,\la_\std,\phi_\std)$ as a Weinstein domain in such a way that part of the boundary is identified with $M \times S^1$. 
To accomplish this, Weinstein homotope the completion $(\wh{M},\wh{\la},\wh{\phi})$ to $(\wh{M},\wh{\la},\wt{\phi})$, where $\wt{\phi}: \wh{M} \rightarrow \R$ is $C^\infty$-small on $M$ and is of the form $h(s)$ on $[0,\infty) \times \bdy M$, where $s$ is the coordinate on $[0,\infty)$ and $h'(s) > 0$.
Now consider the Weinstein manifold $(\wh{M}\times \C, \wh{\la}+\frac{1}{2}xdy - \frac{1}{2}ydx, \wt{\phi} + x^2 + y^2)$. Denote the Weinstein domain given by the sublevel set $\{\wt{\phi} + x^2 + y^2 \leq 1\}$ as $(M,\la,\phi) \ttimes (\D^2,\la_\std,\phi_\std)$ (or just $M \ttimes \D^2$ when the rest of the data is implicit).

We can decompose the contact boundary $Y := \bdy(M\ttimes \D^2)$ into two parts:
\begin{align*}
Y_1 := Y \cap (M \times \C)\\
Y_2 := Y \setminus \Int Y_1.
\end{align*}
Observe that $Y_1$ is diffeomorphic to $M \times S^1$, and the contact structure is given by
\begin{align*}
\la + \tfrac{1}{2}(1 - \wt{\phi})dt,
\end{align*}
where $t$ is the coordinate on $S^1 = \R/(2\pi\Z)$.
We would like to remove the $\wt{\phi}$ term. To this effect, using the smallness assumption on $\wt{\phi}|_{M}$
we can find a contact form on $\bdy(M\ttimes\D^2)$ which is given by $\la + \tfrac{1}{2}dt$ on $Y_1$ and agrees with $(\wh{\la}+\frac{1}{2}xdy - \frac{1}{2}ydx)|_{\bdy(M\ttimes\D^2)}$ outside $\Op(Y_1)$.
By Gray's stability theorem, we get a contact embedding 
\begin{align*}
(M \times S^1, \la + \tfrac{1}{2}dt) \hookrightarrow (\Op (Y_1), (\wh{\la}+\tfrac{1}{2}xdy - \tfrac{1}{2}ydx)|_{\Op (Y_1)}).
\end{align*}
Using this embedding, we can assume without loss of generality that the contact form on $Y_1$ is given by $\la + \tfrac{1}{2}dt$.

Now suppose that $V_1,...,V_k \subset (M,\la)$ are a collection of exact Lagrangians with 
$\la|_{V_i} = dF_i$ for $i = 1,...,k$. If $|| F_i|| < \frac{\pi}{2k}$ for $i = 1,...,k$ we can lift $V_1,...,V_k$ to disjoint Legendrian spheres in $Y_1$. More precisely, we can find Legendrian spheres $\La(V_1),...,\La(V_k) \subset Y_1$ such that
\begin{itemize}
\item
for $i = 1,..,k$, $\La(V_i)$ projects diffeomorphically onto $V_i$ under the projection $Y_1 \rightarrow M$
\item for $i = 1,...,k$, the projection of $\La(V_i)$ to the $S^1$-factor is contained in an arc $\Delta_i \subset S^1$ such that $\Delta_i \cap \Delta_j = \emptyset$ if $i \neq j$ and $\Delta_1,...,\Delta_k$ are cyclically ordered in the counter-clockwise $S^1$ direction.
\end{itemize}
Indeed, we take the Legendrian lift of $V_i$ to be
\begin{align*}
\La(V_i) := \{(p,t) \in Y_1\;:\;  p \in V_i,\; t = -2F_i(p) + (i-1)2\pi/k\}.
\end{align*}
Finally, if the condition $|| F_i || < \frac{\pi}{2k}$ is not satisfied, we can simply replace $V_i$ by its image under the Liouville flow of $W$ for large backwards time, which is Hamiltonian isotopic to $V_i$.
We now set $\fW(M,\la,\phi;V_1,...,V_k)$ to be the Weinstein domain given by attaching Weinstein handles to $M \ttimes \D^2$ along the Legendrian lifts $\La(V_1),...,\La(V_k) \subset \bdy(M \ttimes \D^2)$.
\begin{definition}
A {\em Lefschetz presentation} for a Weinstein domain $(X^{2n+2},\la,\phi)$ is a Weinstein deformation equivalence between $(X^{2n+2},\la,\phi)$ and $\fW(M^{2n},\theta,\psi;V_1,...,V_k)$ for some {\em Weinstein Lefschetz datum} $(M^{2n},\theta,\psi;V_1,...,V_k)$. 
\end{definition}
\begin{remark}
The cyclic ordering of the vanishing cycles is an essential part of the Lefschetz data - changing the cyclic ordering often completely changes the symplectic topology (and even the smooth topology) of the total space.
\end{remark}

\begin{definition}\label{def:stabilization}
Let $\fW(M^{2n},\theta,\psi;V_1,\ldots,V_k)$ be a Weinstein Lefschetz fibration, and let $D^n \sse M$ be a Lagrangian disk whose boundary is a Legendrian sphere in $\dd M$. Let $H$ be a Weinstein handle attached to $M$ along $\dd D$, and let $S \sse M \cup H$ be the Lagrangian sphere formed by the union of $D^n$ and the core of $H$.
The Weinstein Lefschetz fibration $\fW(M \cup H,\theta,\psi;V_1,\ldots,V_k, S)$ is called the \emph{stabilization} of $\fW(M,\theta,\psi;V_1,\ldots,V_k)$ along $D$, which is called a \emph{stabilizing disk} (see for example \cite{van2010lecture}).
\end{definition}

\begin{remark}
As is well known, the Weinstein deformation type of the total space of a Lefschetz fibration does not change with stabilization. One can see this by viewing $S$ and $\dd D$ as the attaching maps of a canceling pair of Weinstein handles. Alternatively, it can be seen as boundary connect summing with $\D^{2n+2}_\std$, which has the Lefschetz presentation $\fW(T^*S^n,\lambda_\std,\psi;Z)$, where $Z$ is the zero section, and $\psi$ is a Morse perturbation of the function $|p|^2$.
\end{remark}

To apply subflexibilization to a Lefschetz fibration, we will need to make the following additional assumption on the vanishing cycles $V_1,...,V_k$.
\begin{assumption}\label{mainassumption}
There are Lagrangian disks $T_1,...,T_k \subset M$ with disjoint Legendrian boundaries in $\bdy M$ such that, for each $1 \leq i \leq k$, $T_i$ intersects $V_i$ transversely in a single point (however $T_i$ is allowed to intersect the other vanishing cycles arbitrarily). 
\end{assumption}
This is the same as Assumption 3.6 in \cite{abouzaid2010altering}. 
In fact, in order to apply the results from \cite{Siegel}, we will make a slightly stronger assumption as follows.
Consider an embedded path $\mu$ in the base of a Lefschetz fibration which intersects the critical values precisely at its endpoints. In the fiber above the midpoint of the path we get two vanishing cycles corresponding to the critical values at the two endpoints.
If these two (parametrized) Lagrangian spheres are Hamiltonian isotopic, we call $\mu$ a {\em matching path}, and we can construct a corresponding {\em matching cycle}, a Lagrangian sphere in the total space which projects to $\mu$ (see \cite[\S 16g]{seidelbook} for details).
\begin{definition}
A Weinstein Lefschetz datum $(M,\theta,\psi;V_1,...,V_k)$ is of {\em matching type} if $(M,\theta,\psi)$
itself admits Lefschetz presentation such that each $V_i$ is a matching cycle with respect to this auxiliary Lefschetz fibration.
\end{definition}
It is explained in Example 3.13 of \cite{abouzaid2010altering} that every smooth complex affine variety (viewed as a Weinstein domain) admits a Lefschetz presentation of matching type, and any matching type Lefschetz fibration satisfies Assumption \ref{mainassumption}.
\begin{remark}
By recent work of Giroux--Pardon \cite{giroux2014existence}, every Weinstein domain admits a Lefschetz presentation.
\end{remark}

\section{Holomorphic curve invariants and (sub)flexibility}\label{sec: holomorphic curve invariants}

Our main goal in this section is to prove that (twisted) symplectic cohomology vanishes for any flexible Weinstein domain.
To review symplectic cohomology and its twisted variant, we refer the reader to \cite{Siegel} and the references cited therein. 

Let $(W,\la)$ be a Liouville domain and let $(\wh{W},\wh{\la})$ denote its completion.
Working over any ground ring $\K$, we have $\sh(W,\la)$, the symplectic cohomology of $(W,\la)$.
Among other things, this is a unital $\K$-algebra which is invariant under symplectomorphisms (not necessarily exact) of $(\wh{W},\wh{\la})$ (see the discussion in \cite[\S 2c]{abouzaid2010altering}).
Informally, the symplectic cohomology of $(W,\la)$ is defined as the Hamiltonian Floer cohomology of $(\wh{W},\wh{\la})$ with respect to a Hamiltonian which grows sufficiently rapidly at infinity. 
In particular, the underlying chain complex is generated by $1$-periodic Hamiltonian orbits in $\wh{W}$, and the differential counts isolated solutions to Floer's equation.
Symbolically, the differential applied to an orbit $\gamma_+$ is of the form
\begin{align*}
\delta(\gamma_+) := \sum_{\substack{u \in \calM(\gamma_-,\gamma_+)^\oo}} \signu \gamma_-,
\end{align*}
where the sum is over all orbits $\gamma_-$. Here $\calM(\gamma_-,\gamma_+)^\oo$ denotes the moduli space of isolated (unparameterized) Floer trajectories asymptotic to $\gamma_-, \gamma_+$, and $\signu \in \{\pm 1\}$ is a certain associated sign.
Of course, one must work carefully to ensure that the Gromov's compactness theorem applies, and this necessitates picking Hamiltonians and almost complex structures of a special form at infinity.

From now on we assume for simplicity that $\K$ is a field.
In order to apply the twisting construction, we further assume there is an injective group homomorphism $T: \R \rightarrow \K^*$ from the additive group $\R$ to the group of invertible elements in $\K$. We set $t := T(1)$ and more generally $t^r := T(r)$ for any $r \in \R$. 
For example, we could take $\K$ to be simply $\R$ with $T(r) := e^r$. Or, more formally, we could take $\K$ to be the field of rational functions in a formal variable $t$ with real exponents and coefficients in say $\Z/2$.
Note that in this paper there is no need to take a Novikov completion since we are working in an exact setting.

\begin{remark}
In the case $c_1(W,\omega) = 0$, we can trivialize the canonical bundle of $(W,\omega)$, in which case $\sh(W,\la)$ inherits a $\Z$-grading via the Conley--Zehnder index. 
This grading plays a role in Proposition \ref{prop:subcrit handle attach} below.
\end{remark}

Now suppose that $\Omega$ is a closed two-form on $W$.
In this case we can define $\sh_\Omega(W,\la)$, the symplectic cohomology of $(W,\la)$ twisted by $\Omega$.
This is again a unital $\K$-algebra, depending only on the cohomology class of $\Omega$ and invariant under symplectomorphisms of $(\wh{W},\wh{\la})$ which respect $[\Omega] \in H^2(\wh{W};\R)$.
The definition of $\sh_\Omega(W,\la)$ is almost identical to that of $\sh(W,\la)$, except that in defining the differential and other relevant structure maps
we weight each curve $u: \R \times S^1 \rightarrow \wh{W}$ by the factor $t^{\int u^*\Omega}$.
For example, the twisted differential is of the form 
\begin{align*}
\delta_\Omega(\gamma_+) := \sum_{\substack{u \in \calM(\gamma_-,\gamma_+)^\oo}} t^{\int u^*\Omega}\signu \gamma_-.
\end{align*}

\sss

Our present goal is to prove the following:
\begin{thm}\label{thm:vanishing SH}
For $(W,\la,\phi)$ a flexible Weinstein domain, we have $\sh_\Omega(W,\la) = 0$ for any closed two-form $\Omega$ on $W$. 
\end{thm}
At least for untwisted symplectic cohomology, this result is already well-known to experts.
One argument uses the surgery results of \cite{bourgeois2012effect}, together with the triviality of Legendrian contact homology for loose Legendrians. 
It would be straightforward to extend the techniques of \cite{bourgeois2012effect} to make this argument work in the twisted case as well.
Instead, we give a proof of Theorem \ref{thm:vanishing SH} based on an embedding h-principle for flexible Weinstein domains (see Theorem \ref{thm:subflex embed} below).

Firstly, we will need the K\"unneth theorem for symplectic cohomology, adapted to the twisted case. 
An inspection of Oancea's proof \cite{oancea2006kunneth} in the untwisted setting shows that the twisting two-forms benignly come along for the ride,
yielding the following.
\begin{thm}\label{thm:kunneth}
Let $\Omega_1$ and $\Omega_2$ be closed two-forms on Liouville manifolds $(W_1,\la_1)$ and $(W_2,\la_2)$ respectively.
There is an isomorphism
\begin{align*}
\sh_{\Omega_1}(W_1,\la_1) \otimes \sh_{\Omega_2}(W_2,\la_2) \cong \sh_{\Omega_1+\Omega_2}(W_1 \times W_2,\la_1 + \la_2).
\end{align*}
\end{thm}
Recall that Oancea's basic idea is to consider a split Hamiltonian and almost complex structure on $\wh{W_1} \times \wh{W_2}$, for which the resulting Floer cohomology
is as expected via the algebraic K\"unneth theorem. 
However, this does not compute symplectic cohomology since this Floer data is not cylindrical at infinity with respect to the Liouville form on $\wh{W_1} \times \wh{W_2}$.
To rectify this, Oancea carefully modifies the split data to make it cylindrical at infinity and argues via action considerations that this procedure does not change the resulting homology.
Alternatively, an unpublished argument due to Mark McLean starts with cylindrical Floer data and modifies this to the product data. In this case one can appeal to the more robust integrated maximum principle of Abouzaid to rule out undesired Floer trajectories.

\sss

Next, we need to understand how flexibility behaves under products.
\begin{prop}\label{prop:products are flexible}
Let $(W,\la,\phi)$ be an explicitly flexible Weinstein domain and let $(W',\la',\phi')$ be any Weinstein domain. Then the product 
$(\wh{W} \times \wh{W'},\wh{\la} + \wh{\la'},\wh{\phi} + \wh{\phi'})$ is an explicitly flexible Weinstein manifold.
\end{prop}
\begin{proof}
The critical values of $\wh{\phi} + \wh{\phi'}$ are of the form $q + q'$ for $q$ a critical value of $\phi$ and $q'$ a critical value of $\phi'$.
Consider some pair $q,q'$ of critical values of $\phi$ and $\phi'$ respectively.
For sufficiently small $\e > 0$, set
\begin{itemize}
\item
$N := (\wh{\phi} + \wh{\phi'})^{-1}(q+q'-\e)$, equipped with the contact form $\alpha_N := (\wh{\la} + \wh{\la'})|_N$
\item 
$Q := \phi^{-1}(q - \e)$, equipped with the contact form $\alpha_Q := \la|_Q$
\item
$Q' := (\phi')^{-1}(q' - \e)$, equipped with the contact form $\alpha_{Q'} := \la'|_{Q'}$.
\end{itemize}
Let $\dtop(q) \subset W$ denote the union of the descending manifolds of all critical points of $\phi$ with critical value $q$ and index equal to 
half the dimension of $W$. Define $\dtop(q') \subset W'$ and $\dtop(q+q') \subset \wh{W} \times \wh{W'}$ similarly.
We have Legendrian links
\begin{itemize}
\item
$\La_N := \dtop(q+q') \cap N$ in $(N,\ker \alpha_N)$
\item
$\La_Q := \dtop(q) \cap Q$ in $(Q,\ker \alpha_Q)$
\item
$\La_{Q'} := \dtop(q') \cap Q'$ in $(Q',\ker \alpha_{Q'})$.
\end{itemize}
As explained in Remark \ref{remark:explicitflexcrit}(3), explicit flexibility of $(W,\la,\phi)$ means that $\La_Q$ is a loose Legendrian link for each critical value $q$ of $\phi$,
and our goal is to prove that $\La_N$ is loose for each critical value $q+q'$ of $\wh{\phi} + \wh{\phi'}$.

Using the Liouville flow of $\wh{\la}$, we get a smooth embedding $\Phi: \R_s \times Q \hookrightarrow \wh{W}$ such that
$\Phi^*(\wh{\la}) = e^s\alpha_Q$. Here we are using the fact that the Liouville vector field $Z_{\wh{\la}}$ on $\wh{W}$ is complete and exponentially expands the Liouville one-form $\wh{\la}$.
Similarly, we get an embedding $\Phi': \R_{s'} \times Q' \hookrightarrow \wh{W'}$ such that $(\Phi')^*(\wh{\la'}) = e^{s'}\alpha_{Q'}$.
Consider the hypersurface $$H := \{s + s' = 0\} \subset \R_s \times Q \times \R_{s'} \times Q',$$
which comes with a codimension one embedding $$i := (\Phi \times \Phi')|_H: H \hookrightarrow \wh{W} \times \wh{W'}$$ transverse to the Liouville vector field $Z_{\wh{\la} + \wh{\la'}}$.
Identifying $H$ with $\R_s \times Q \times Q'$, where $s'$ is determined by the relation $s' = -s$,
$H$ naturally inherits the contact one-form 
$$\alpha_H := i^*(\wh{\la} + \wh{\la'}) = e^s\alpha_Q + e^{-s}\alpha_{Q'},$$ along with the Legendrian submanifold
$$\La_H :=  i^{-1}(\dtop(q) \times \dtop(q')) = \R_s \times \La_Q \times \La_{Q'}.$$

By Lemma \ref{lemma:looseproductwithsymplectization} below, $\La_H$ is loose.
Indeed, the ambient contact structure on $H$ is of the form $\ker(\alpha_Q + e^{-2s}\alpha_{Q'})$.
In particular, the factor $\R_s \times Q'$ is equipped with the one-form $e^{-2s}\alpha_{Q'}$,
for which the dual vector field $-\tfrac{1}{2}\bdy_s$ is everywhere nonzero and tangent to the Lagrangian submanifold $\R_s \times \Lambda_{Q'}$.
This means we can find loose charts for $\La_H$ in $(H,\ker \alpha_H)$, one for each component in the complement of the other components.

We claim that these loose charts for $\La_H$ can be pushed forward under the Liouville flow of $\wh{\la} + \wh{\la'}$  
to loose charts for $\La_N$ in $(N,\ker \alpha_N)$.
Indeed, let $F_t: \wh{W} \times \wh{W'} \rightarrow \wh{W} \times \wh{W'}$ denote the time-$t$ Liouville flow for $t \in \R$. Since there are no critical values of $\phi$ in $[q-\e,q)$ or of $\phi'$ in $[q'-\e,q')$, for each $x \in H$ we have
${\underset{t \rightarrow -\infty}\lim(\wh{\phi} + \wh{\phi'})(F_t(i(x))) < q + q' - \e}$ and ${\underset{t \rightarrow +\infty}\lim(\wh{\phi} + \wh{\phi'})(F_t(i(x))) \geq q + q'}$.
Since $\wh{\phi} + \wh{\phi'}$ is increasing under the Liouville flow,
there is a uniquely determined ``flow time" function $T: H \rightarrow \R$ 
such that the map sending $x \in H$ to $F_{T(x)}(i(x))$ defines a smooth embedding $F_T: H \hookrightarrow N$
and satisfies $F_T^*(\alpha_N) = e^T\alpha_H$.
In particular, $F_T$ is a contact embedding
and maps $\La_H$ to $\La_N$ since the Liouville vector field is tangent to $\dtop(q+q')$.
 Therefore it sends the loose charts for $\La_H$ in $(H,\ker\alpha_H)$ to 
corresponding loose charts for $\La_N$ in $(N,\ker \alpha_N)$.

Strictly speaking, the above shows that $F_T(\La_H)$ is a loose link in $(N,\ker\alpha_N)$, and in principle there could be duplicates in the collection $\{q + q'\}$ where $q$ and $q'$ vary over the critical values of $\phi$ and $\phi'$ respectively. Therefore we need to check that 
the loose charts we produce for $\La_N$ in $(N,\ker\alpha_N)$ for different pairs $(q,q')$ are disjoint.
To see this, suppose that $r,r'$ is a distinct pair of critical values of $\phi,\phi'$ respectively, with $r + r' = q + q'$.
We claim that $\dtop(r) \times \dtop(r')$ is disjoint from $i(H)$. Since $\dtop(r) \times\dtop(r')$ is Liouville flow invariant, this will imply that it is also disjoint from $F_T(H)$, 
and hence from the above loose charts for $\La_H$ in $H$ corresponding to $(q,q')$. 
In fact, observe that $\dtop(r) \times \dtop(r')$ is disjoint from the image of $\Phi \times \Phi'$.
Indeed, if $r < q$, we must have that $\dtop(r)$ is disjoint from $Q$, since upward flow trajectories ending at level $r$ will never reach level $q - \e$. Otherwise, if $r' < q'$, we must have that $\dtop(r')$ is disjoint from $Q'$.
Either way, by invariance with respect to the flows of $Z_{\wh{\la}}$ and $Z_{\wh{\la}'}$ we have that $\dtop(r)\times\dtop(r')$ is disjoint from $\Phi(\{0\}\times Q) \times \Phi'(\{0\} \times Q')$, and the claim follows.

\end{proof}

\begin{lemma}\label{lemma:looseproductwithsymplectization}
Let $\La$ be a loose Legendrian link in a closed contact manifold $(Q,\ker\alpha)$, and let $L$ be a Lagrangian in an exact symplectic manifold $(M,\theta)$.
Assume that $\theta$ vanishes when restricted to $L$ and that the dual vector field $Z_\theta$ has flow defined for all time and is nonvanishing along $L$ and tangent to $L$.
Then $\La \times L$ is a loose Legendrian in the contact manifold $(Q \times M, \ker(\alpha + \theta))$.
\end{lemma}

\begin{proof}
Recall from \S\ref{subsec:looseness and flexibility} that $V_r \subset T^*\R^n$ is the box of radius $r$, typically equipped with the canonical Liouville one-form $\lacan = -\sum_{i=1}^np_idq_i$.
Observe that the Liouville vector field $Z_{\lacan} = \sum_{i=1}^n p_i\bdy_{p_i}$ vanishes identically on $Z_r = \{p_1 = ... = p_n = 0\} \subset V_r$, whereas by assumption this is not the case for $Z_\theta$ on $L$, and therefore some preparation is needed before we can appeal to Lemma \ref{lemma:wide nbd} below\footnote{We thank the anonymous referee for pointing out a mistake related to this point in an earlier draft and suggesting a fix.}.
 
Consider the modified Liouville one-form on $V_r$, given by $\lamod = dp_n - \sum_{i=1}^n p_idq_i$.
Notice that the corresponding Liouville vector field $Z_{\lamod} = \bdy_{q_n} + \sum_{i=1}^n p_i\bdy_{p_i}$ is tangent and nonvanishing along $Z_r$.
We claim that we can find an embedding $F: V_r \hookrightarrow M$ such that $F^*\theta = \lamod$ and $F(Z_r) \subset L$, for some sufficiently small $r > 0$.
Indeed, note that $\lamod$ restricts to a contact form on $\{q_n = 0\} \subset V_r$, with respect to which
$\{q_n = 0\} \cap Z_r$ is Legendrian.
Similarly, for a point $x \in L$, we can find a small codimension one submanifold $H \subset \Op(x)$ which is tranverse to $Z_\theta$.
Then $\theta|_H$ is a contact form on $H$, with respect to which $L \cap H$ is Legendrian.
By the contact Darboux thereom (see \cite[Proposition 6.19]{cieliebak2012stein}),
we can find an embedding $f: \{q_n = 0\} \cap V_r \hookrightarrow H$ such that
$f(\{q_n = 0\} \cap Z_r) \subset L \cap H$ and $f^*(\theta|_H) = \lamod|_{ \{q_n = 0\} \cap V_r }$. 
Then the map $F$ is obtained by extending $f$ to be equivariant with respect to the flows of $Z_{\lamod}$ and $Z_{\theta}$ and possibly restricting to a smaller $r > 0$.

Similarly, consider a collection of points $x_1,...,x_k \in L$, one on each connected component $L_i$ of $L$.
Applying the argument above, we can find embeddings $F_i: V_r \hookrightarrow M$ for $i = 1,...,k$ such that
  $F_i^*\theta = \lamod$ and $F_i(Z_r) \subset L_i$.
We also arrange that the image of $F_i$ is disjoint from $L \setminus L_i$.

Fix some $C > 0$ sufficiently large relative to $r$.
The flow of $Z_\theta$ for time $\log(C)$ gives a diffeomorphism $M \rightarrow M$ which pulls back $\theta$ to $C\theta$. 
Then the composition given by $F_i$ followed by this flow is a map $\wt{F_i}: V_r \hookrightarrow M$ such that $\wt{F_i}^*\theta = C\lamod$.
Since $Z_\theta$ is tangent to $L$, we also have $\wt{F_i}(Z_r) \subset L_i$.

Next, consider a function $T: V_r \rightarrow \R$,
and let $G_T: Q \times V_r \rightarrow Q \times V_r$ be the map given by the time-$T$ Reeb flow of $\alpha$ on the first factor and the identity on the second factor.
Then we have $G_T^*(\alpha + C\lamod) = \alpha + C\lamod + dT$.
In particular, setting $T = -Cp_n$, we get a map
$G := G_T$ which satisfies $G^*(\alpha + C\lamod) = \alpha + C\lamod - Cdp_n = \alpha + C\lacan$.
Since $T$ vanishes on $\La$, we also have $G(\La \times Z_r) = \La \times Z_r$.

Finally, the composition $(\op{Id} \times \wt{F_i}) \circ G: Q \times V_r \hookrightarrow Q \times M$ maps $\La \times Z_r$ to $\La \times L_i$ and pulls back $\alpha + \theta$ to $\alpha + C\lacan$. 
By rescaling, we can identify $(V_r,C\lacan)$ with $(V_R,\lacan)$ for $R = \sqrt{C}r$.
In particular, since $R$ is arbitrarily large, it follows from Lemma \ref{lemma:wide nbd} below that $\La \times L$ is loose.
\end{proof}

\begin{lemma}\label{lemma:wide nbd}
Let $\La$ be a loose Legendrian link in a contact manifold $(N^{2l+1},\ker \alpha)$.
Then $\La \times Z_r^{n}$ is a loose Legendrian link in the contact manifold
${(N^{2l+1} \times V_r^{2n},\ker(\alpha - p_1dq_1 - ... - p_ndq_n))}$, provided $r$ is sufficiently large.
\end{lemma}

\begin{proof}
Since $\La$ is loose, for each connected component there is a contact embedding of pairs
\begin{align*}
G: (B^3_\std \times V_\rho^{2l-2}, \gamma_a \times Z_\rho^{l-1}) \hookrightarrow (N^{2l+1},\La)
\end{align*}
for some $\rho>0$ with $a < 2\rho^2$,
with image disjoint from the other components of $\La$.
This means that $G^*\alpha$ is equal to $z - ydx - p_1dq_1 - ... - p_{l-1}dq_{l-1}$ 
times $e^f$ for some smooth function $f$.
Then for each component of $\La$ we can define a contact embedding of pairs
\begin{align*}
(B^3_\std \times V_\rho^{2l-2} \times V_\rho^{2n},\gamma_a \times Z_\rho^{l-1} \times Z_\rho^n) \hookrightarrow (N^{2l+1} \times V_r^{2n},\La \times Z_r^n)
\end{align*}
given by $G$ on the $B^3_\std \times V_\rho^{2l-2}$ factor and the time $f$ Liouville flow on the remaining factor.
Note that this map is well-defined for $r$ sufficiently large, and the image is disjoint from the other components of $\La \times Z_r^n$. This restricts to a loose chart
\begin{align*}
(B^3_\std \times V_\rho^{2l-2+2n},\gamma_a \times Z_\rho^{l-1+n}) \hookrightarrow (N^{2l+1} \times V_r^{2n},\La \times Z_r^n).
\end{align*}
\end{proof}

Finally, we have the following corollary of the Lagrangian caps h-principle, amounting to an h-principle for symplectic embeddings of flexible Weinstein domains.
\begin{thm}{\normalfont (Eliashberg--Murphy \cite[\S 6]{eliashberg2013lagrangian})} \label{thm:subflex embed}
Let $(W^{2n}, \lambda, \p)$ be a flexible Weinstein domain, and let $F: W \hookrightarrow (X^{2n}, \omega_X)$ be any smooth codimension zero embedding into a symplectic manifold. If $n=3$, assume $X$ has infinite Gromov width.\footnote{This extra assumption on the $n=3$ case proved to be unnecessary, as was shown by T. Yoshiyasu \cite{Yoshiyasu}.}
Suppose that $F^*\omega_X$ is an exact two-form and is homotopic to $d\la$ through non-degenerate two-forms.
Then $F$ is isotopic to a symplectic embedding $f: (W, \e d\lambda) \hookrightarrow (X, \omega_X)$, for some sufficiently small $\e >0$.
If $(X,\omega_X)$ is an exact symplectic manifold, then we can also arrange that $f$ is an exact symplectic embedding.
\end{thm}

\noindent \textit{Proof of Theorem \ref{thm:vanishing SH}}: 
Let $(W,\la,\phi)$ be a flexible Weinstein domain, and let $(D^*S^1,\la_\can)$ denote the unit disk cotangent bundle of $S^1$, equipped with its canonical Liouville form. By Theorem \ref{thm:kunneth}, we have 
\begin{align*}
\sh_{\Omega}(W,\la) \otimes \sh(D^*S^1,\la_\can) \cong \sh_{\Omega}(W \times D^*S^1,\la+\la_\can)
\end{align*}
(here $\Omega$ is being abusively used to also denote its pullback to $W \times D^*S^1$, and the twisting two-form on $D^*S^1$ is trivial).
Recall that $\sh(D^*S^1,\la_\can)$ is isomorphic to the homology of the free loop space of $S^1$ (see \cite{abouzaid2013symplectic}) and in particular is nontrivial.
Therefore it suffices to prove that $\sh_{\Omega}(W \times D^*S^1,\la+\la_\can)$ is trivial.

Let $(\D^2_r,\la_\std)$ denote the standard Liouville disk of radius $r >0$. Consider a split symplectic embedding 
\begin{align*}
i: W \times D^*S^1 \hookrightarrow W \times \D^2_r,
\end{align*}
given by 
\begin{itemize}
\item
the identity on the first factor
\item any area-preserving embedding on the second factor.
\end{itemize}
Note that this embedding is not exact.
Nevertheless, since the product is flexible by Proposition \ref{prop:products are flexible},
 it satisfies the hypotheses of Theorem \ref{thm:subflex embed}.
Therefore $i$ is isotopic to a Liouville embedding
\begin{align*}
i': (W \times D^*S^1, \la + \la_\can) \hookrightarrow (W \times \D^2_r, \la + \la_\std).
\end{align*}
Moreover, $i'_*\Omega$ extends to a closed two-form $\wt{\Omega}$ on $W \times \D_r^2$.
Note that without changing the cohomology class, we can assume that $\wt{\Omega}$ is pullback of a closed two-form on $W$.
Another application of the Theorem \ref{thm:kunneth}, together with the vanishing of $\sh(\D^2_r,\la_\std)$,
implies that $\sh_{\wt{\Omega}}(W\times \D^2_r,\la + \la_\std)$ vanishes. 
Finally, a well-known argument involving the Viterbo transfer map (see Theorem \ref{thm:transfer} below) shows that ${\sh_{\Omega}(W \times D^*S^1,\la + \la_\can)}$ is trivial.
Namely, the transfer map 
\begin{align*}
\sh_{\Omega}(W \times \D^2_r, \la + \la_\std) \rightarrow \sh_{\Omega}(W \times D^*S^1,\la+\la_\can)
\end{align*}
is a unital ring map, which forces $\sh_{\Omega}(W \times D^*S^1,\la+\la_\can)$ to be the trivial ring with unit equal to zero. 
$\hfill\Box$

\sss

One very useful feature of symplectic cohomology is the Viterbo transfer map. This also holds for twisted symplectic cohomology, provided we take care in how we twist. 
\begin{thm}{\normalfont (\cite{viterbo1999functors}, see also \cite{ritter2013topological})}\label{thm:transfer}
Let $(W_0,\la_0)$ and $(W,\la)$ be Liouville domains, and assume there is a Liouville embedding $i: (W_0,\la_0) \hookrightarrow (W,\la)$.
For $\Omega$ any closed two-form on $W$, there is map of unital $\K$-algebras
\begin{align*}
\sh_{\Omega}(W,\la) \rightarrow \sh_{i^*\Omega}(W_0,\la_0).
\end{align*}
\end{thm}
\begin{coro}\label{coro:transfer}
$\sh(W,\la) = 0$ for any subflexible Weinstein domain $(W,\la,\phi)$.
\end{coro}
However, we wish to point out that there could be a closed two-form $\Omega_0$ on $W_0$ which makes $\sh_{\Omega_0}(W_0,\la_0)$ nontrivial, but such that $\Omega_0$ does not extend to $W$ as a closed two-form. Indeed, this is precisely what happens for the examples we construct in $\S \ref{sec:subflexibilization}$.

\section{Subflexibilization}\label{sec:subflexibilization}

\subsection{Maydanskiy's manifold revisited}\label{subsec:Maydanskiy}

Before giving the general subflexibilization construction we illustrate the main ideas with an important example.
We consider the two $6$-dimensional Weinstein domains $X_1^6$ and $X_2^6$ from \cite{harris2012distinguishing}. 
Each is diffeomorphic to $D^*S^3 \bcs (D^*S^2 \times \D^2)$ and is represented by a Weinstein Lefschetz presentation.
Here $M_1 \bcs M_2$ denotes the {\em boundary connect sum} of two equidimensional Weinstein domains $M_1$ and $M_2$, i.e. the result of attaching a Weinstein $1$-handle to the disjoint union of $M_1$ and $M_2$ with one endpoint on $\bdy M_1$ and the other on $\bdy M_2$.
In the case of $X_1$, the fiber of the Lefschetz fibration is the Milnor fiber $A_2^4$, i.e. a plumbing of two copies of $D^*S^2$. 
The fiber $A_2^4$ itself admits a secondary Lefschetz fibration with fiber $D^*S^1$ and three vanishing cycles. 
The Lefschetz presentation for $X_1$ has two vanishing cycles $V_1,V_2 \subset A_2^4$, given by the matching cycle construction with the matching paths in the top left of Figure \ref{X1andX2}. 
Similarly, the Lefschetz presentation for $X_2$ has fiber $A_4^4$ (a plumbing of four copies of $D^*S^2$), which admits a secondary Lefschetz fibration with fiber $D^*S^1$ and five vanishing cycles.
$X_2$ has four vanishing cycles $W_1,W_2,W_3,W_4 \subset A_4^4$ specified by the matching paths in the bottom left of Figure \ref{X1andX2}. 
We refer the reader to \cite{maydanskiy2009,harris2012distinguishing} for more details.

\begin{figure}
 \centering
 \includegraphics[scale=.8]{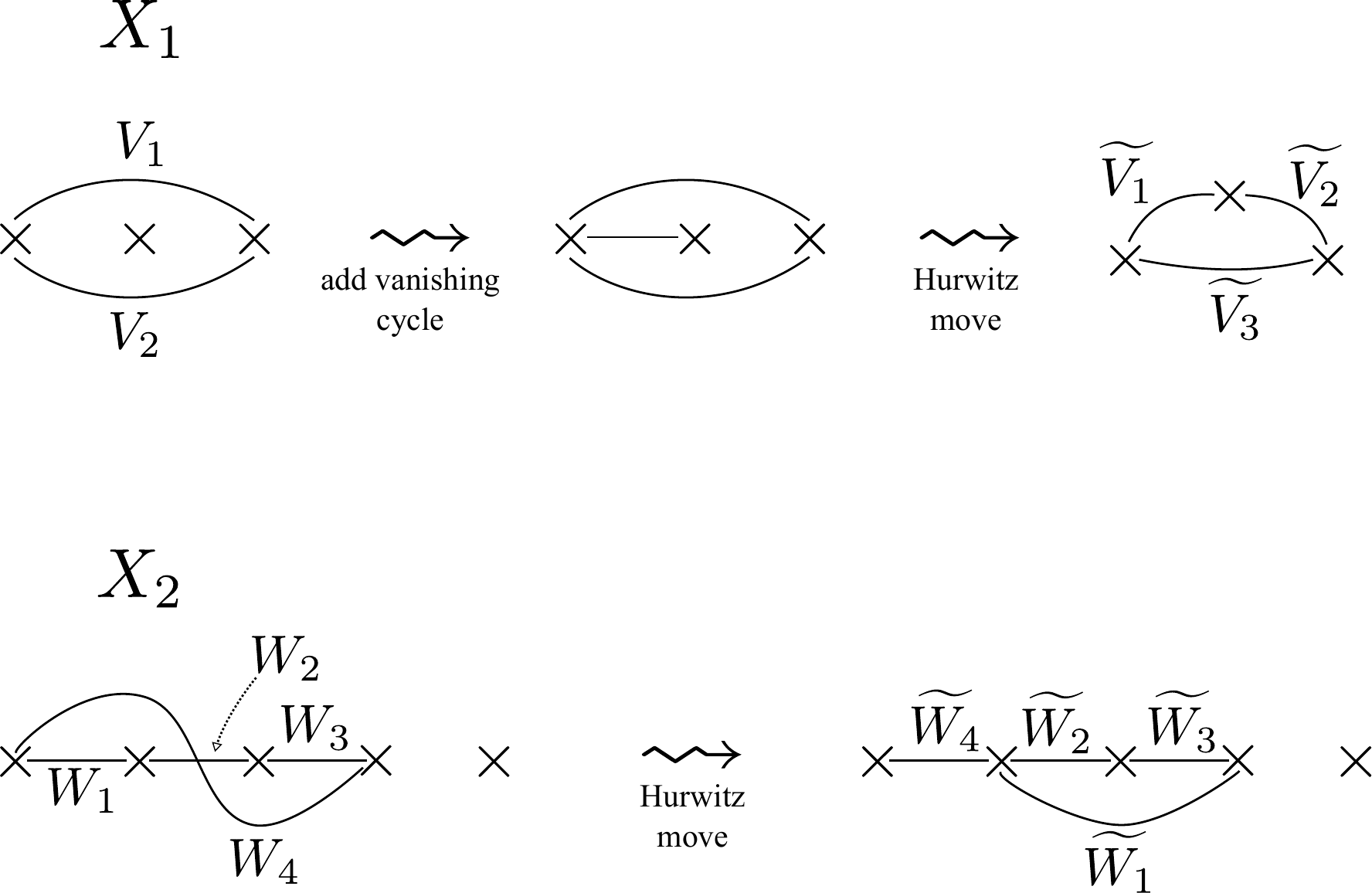}
 \caption{Top: $X_1$ becomes flexible after adding a vanishing cycle and applying a Hurwitz move. Bottom: $X_2$ becomes flexible after applying a Hurwitz move.}
 \label{X1andX2}
\end{figure}  

The manifold $X_1$ was first considered by Maydanskiy in \cite{maydanskiy2009}, where he applies Seidel's long exact sequence to show that $X_1$ has vanishing wrapped Fukaya category (and hence vanishing symplectic cohomology). This also follows from the fact that $X_1$ is {\em subflexible}, which can be seen as follows. Adding a vanishing cycle to  $X_1$ and applying a Hurwitz move according to the top of Figure \ref{X1andX2}, we arrive at the Weinstein domain 	
$ \fW(A_2^4;\wt{V_1},\wt{V_2},\tau_{\wt{V_2}}\wt{V_1})$, which we show is flexible below 
(Theorem \ref{subflexthm}).  
Incidentally, $\fW(A_2^4;\wt{V_1},\wt{V_2},\tau_{\wt{V_2}}\wt{V_1})$ is precisely the exotic $D^*S^3$ illustrated in \cite[Figure 2]{maydanskiyseidel},
which gives another way of seeing that its symplectic cohomology vanishes.

On the other hand, $X_2$ is actually flexible. To see this, apply the Hurwitz move illustrated in the bottom of Figure \ref{X1andX2}.
The result is Weinstein deformation equivalent to the same flexible $D^*S^3$ as before, with an extra pair of canceling handles ($\wt{W_4}$ is in canceling position) and an extra subcritical handle attached.

The main theorem of \cite{harris2012distinguishing} states that $X_1$ and $X_2$ are not symplectomorphic. Harris argues that an embedded Lagrangian sphere appears in $X_1$ after an arbitrarily small deformation of the symplectic form, whereas no such Lagrangian can appear for $X_2$.
From the perspective of this paper, $X_1$ and $X_2$ can be distinguished by the fact that $X_1$ has nontrivial twisted symplectic cohomology by Theorem \ref{thm:nontrivial twSH}, whereas $X_2$ is flexible and thus has trivial twisted symplectic cohomology. 

The fact that $X_1$ has nontrivial twisted symplectic cohomology reflects the symplectomorphism type of $X_1$ after a small deformation of its symplectic form.
Namely, notice that $X_1$ can be viewed as $\fW(A_2^4;S_1,\tau_{S_2}^2S_1)$, where $S_1,S_2 \subset A_2^4$ represent the plumbed zero sections.
According to Seidel \cite{seidel1997floer}, a squared Dehn twist around a two-dimension sphere is a {\em fragile} 
symplectomorphism, and hence we expect $\tau_{S_2}^2S_1$ to be Hamiltonian isotopic to $S_1$ in the presence of an ambient deformation of the symplectic form.
Since $\fW(A_2^4;S_1,S_1)$ is just $D^*S^3$ with an extra subcritical $2$-handle,
this should imply that after deformation $X_1$ becomes symplectomorphic to a rather trivial deformation of the standard Weinstein structure on $D^*S^3 \bcs (D^*S^2 \times \D^2)$ (this is made precise in Theorem \ref{deformationthm}). By the standard tools in symplectic cohomology we might expect the symplectic cohomology of the latter space to agree with homology of the free loop space of $S^3$.

However, we point out that there are technical difficulties in making sense of the symplectic cohomology of $X_1$ and $X_2$ after deformations. 
This is because $X_1$ and $X_2$ become non-convex symplectic manifolds, and the standard definition and maximum principle properties of symplectic cohomology do not apply to such manifolds. On the other hand, twisted symplectic cohomology for Liouville domains is well-defined and poses no serious technical difficulties beyond those of ordinary symplectic cohomology (although computations are still in general quite difficult).

\subsection{The Subflexibilization construction}\label{subsec:subflex}

Let $X^{2n+2}$ be a Weinstein domain with a Lefschetz presentation satisfying Assumption \ref{mainassumption}.
We explain how to use the Lefschetz presentation to modify $X$ and obtain a new Weinstein domain $\sbf(X)$.

\begin{construction}\label{mainconstruction}
Assume $X^{2n+2}$ has a Weinstein Lefschetz presentation with Weinstein fiber $M^{2n}$, vanishing cycles $V_1,...,V_k \subset M$, and Lagrangian disks $T_1,...,T_k \subset M$ as in Assumption \ref{mainassumption}.
$\sbf(X)$ is defined to have a Weinstein Lefschetz presentation with
\begin{itemize}
\item fiber given by attaching a Weinstein $n$-handle $H_i$ to $M$ along $\bdy T_i$, for each $1 \leq i \leq k$
\item vanishing cycles $\tau_{S_1}^2V_1,\tau_{S_2}^2V_2,...,\tau_{S_k}^2V_k$, where $S_i$ is the Lagrangian sphere given by the union of $T_i$ and the core of $H_i$.
\end{itemize}
We note that $\sbf(X)$ depends on the Lefschetz presentation of $X$, but we suppress this from the notation.
\end{construction}
In the case $\dim X = 6$, let $X_+^6$ denote the boundary connect sum of $X$ with $k$ copies of the subcritical Weinstein domain $D^*S^2 \times \D^2$.
In fact, as explained in \cite{maydanskiyseidelcorrigendum}, $\tau_{S_i}^2V_i$ and $V_i$ are smoothly isotopic through totally real submanifolds (for some choice of compatible almost complex structure) and this implies that the $\sbf(X^6)$ is almost symplectomorphic to 
$\fW(M^4\cup H_1 \cup ... \cup H_k;V_1,...,V_k)$. The handles $H_1,...,H_k$ are attached away from the vanishing cycles and the latter space can be identified with $X_+^6$ (recall that the attaching sphere of each $H_i$ bounds an embedded Lagrangian disk in $M^4$).

\begin{remark}\label{rem:almost symp type}
For $n>2$ the diffeomorphism type of $\sbf(X^{2n+2})$ is not just the straightforward analogue of $X_+^6$.
For $n$ odd, the intersection form on $H_{n+1}(\sbf(X^{2n+2}))$ is different that the one on $H_{n+1}(X^{2n+2})$. This reflects the fact, easily observed by the Picard--Lefschetz formula, that squared Dehn twists around odd dimensional spheres typically act nontrivially on homology.
When $n$ is even and greater than two, the subflexibilization process changes the diffeomorphism type of $X$ in a more subtle way (see \cite{maydanskiyseidelcorrigendum} for more details). 
Since we only need to understand the case $n=2$ for our present applications, we leave it to the reader to work out the topological details in higher dimensions.
\end{remark}

\begin{thm}\label{subflexthm}
For any $X^{2n+2}$ satisfying Assumption \ref{mainassumption}, $\sbf(X)$ is subflexible.
More specifically, when $n=2$ $\sbf(X^6)$ becomes Weinstein deformation equivalent to $\op{Flex}(X^6)$ after attaching $k$ Weinstein $3$-handles.
\end{thm}

The main tool we will use is a proposition from \cite{casals2015geometric} (the proposition there is phrased in terms of open books; boundaries of Lefschetz fibrations are simply a special case of this).

\begin{prop}\label{prop:cmp}

Let $\fW(M,\theta,\psi;V_1,\ldots,V_k)$ be a Lefschetz fibration, let $L \sse M$ be an exact Lagrangian sphere, and let $T \sse M$ be a stabilizing Lagrangian disk which intersects $L$ transversely at one point. Let $H$ be a Weinstein handle attached to $M$ along $\bdy T$,
and let $S$ be the Lagrangian sphere given by the union of $T$ and the core disk of $H$.
Then the Legendrian lift of $\tau_SL$ from the fiber $\pi^{-1}(1)$ of the stabilized Lefschetz presentation $\fW(M \cup H,\theta,\psi;V_1,\ldots,V_k, S)$ is loose in the boundary of the total space. In particular, if the Weinstein domain defined by $\fW(M,\theta,\psi;V_1,\ldots,V_k)$ was flexible, then the Weinstein domain defined by $\fW(M \cup H,\theta,\psi;V_1,\ldots,V_k, S, \tau_SL)$ is flexible as well.

\end{prop}

\begin{proof}[Proof of Theorem \ref{subflexthm}:]
If $X = \fW(M,\theta,\psi;V_1,\ldots,V_k)$, then $$\sbf(X) = \fW(M \cup H_1 \cup \ldots \cup H_k,\theta,\psi;\tau_{S_1}^2V_1,\ldots,\tau_{S_k}^2V_k),$$ where $S_i$ is the Lagrangian sphere given as the union of $T_i$ and the core of the handle $H_i$ (and we continue to use $(\theta, \psi)$ to denote the Weinstein structure on $M \cup H_1 \cup \ldots \cup H_k$). We attach handles to $\sbf(X)$ to define a new Weinstein manifold
$$\wt{X} = \fW(M \cup H_1 \cup \ldots \cup H_k,\theta,\psi;\tau_{S_1}^2V_1,S_1,\tau_{S_2}^2V_2,S_2,\ldots,\tau_{S_k}^2V_k, S_k).$$
We will show that $\wt{X}$ is flexible in any dimension. 
In the case $n=2$, $\wt{X}$ is almost symplectomorphic to 
$$\fW(M \cup H_1 \cup \ldots \cup H_k,\theta,\psi;V_1,S_1,V_2,S_2,\ldots,V_k, S_k),$$
which is just the original manifold $X$ (after stabilizing its Lefschetz fibration $k$ times).
Therefore $\wt{X}$ is almost symplectomorphic to $\op{Flex}(X)$, so once we show that $\wt{X}$ is flexible it follows that it is Weinstein deformation equivalent to $\op{Flex}(X)$ (by Theorem \ref{thm:CE flex}).

\begin{sloppypar}
By applying $k$ Hurwitz moves, we get
$$\wt{X} = \fW(M \cup H_1 \cup \ldots \cup H_k,\theta,\psi;S_1, \tau_{S_1}V_1,\ldots,S_k,\tau_{S_k}V_k).$$
The manifold $\wt{X}_0 := \fW(M, \theta,\psi;\es)$ is subcritical, so in particular it is flexible. ${\fW(M \cup H_1, \theta,\psi; S_1)}$, being a stabilization of the previous Lefschetz fibration, is another Lefschetz presentation of $\wt{X}_0$ (though not explicitly subcritical). ${\wt{X}_1 := \fW(M \cup H_1,\theta,\psi;S_1, \tau_{S_1}V_1)}$ is flexible, because it is built from $X_0$ by attaching a handle to the loose Legendrian $\tau_{S_1}V_1$.
Then $\fW(M \cup H_1 \cup H_2,\theta,\psi;S_1, \tau_{S_1}V_1,S_2)$, being a stabilization of the previous Lefschetz presentation, is Weinstein equivalent to $\wt{X}_1$.  Next $\wt{X}_2 = \fW(M \cup H_1 \cup H_2,\theta,\psi;S_1, \tau_{S_1}V_1,S_2,\tau_{S_2}V_2)$ is flexible, because it is built from the flexible Weinstein manifold $\wt{X}_1$ by attaching a handle along the loose Legendrian sphere $\tau_{S_2}V_2$. Continuing in this way, we see that $\wt{X} = \wt{X}_k$ is flexible.
\end{sloppypar}
\end{proof}

On the other hand, at least when $n=2$, $\sbf(X^6)$ is often not flexible.
Indeed, by Theorem \ref{thm:vanishing SH} it suffices to show that $\sbf(X^6)$ has nontrivial twisted symplectic cohomology. In \cite{Siegel} the following result is proved for any $X$ satisfying Assumption \ref{mainassumption}:
\begin{thm}\label{thm:nontrivial twSH}
There exists a closed two-form $\Omega$ on $X^6$ such that $\sh_{\Omega}(\sbf(X^6)) \cong \sh(X^6)$. 
\end{thm}
\begin{coro}
$\sbf(X^6)$ is not flexible if $\sh(X^6) \neq 0$.
\end{coro}

\sss

The next result describes the symplectomorphism type of $\sbf(X^6)$ after certain small deformations. 
Let $\omega_{sf}$ and $\omega_+$ denote the symplectic forms on $\wh{\sbf(X^6)}$ and $\wh{X^6_+}$ respectively.
\begin{thm}\label{deformationthm}
There are closed two-forms $\Omega_{sf}$ and $\Omega_+$ on $\wh{\sbf(X^6)}$ and $\wh{X^6_+}$ respectively such that  $(\wh{\sbf(X^6)},\omega_{sf} + \e\Omega_{sf})$ and $(\wh{X^6_+},\omega_+ + \e\Omega_+)$ are symplectomorphic for $\e>0$ sufficiently small.  
\end{thm}
\begin{remark}
In \cite{ritter2010deformations}, Ritter proves the isomorphism $\sh_{\Omega}(X,\omega) \cong \sh(X,\omega+\Omega)$ in the case that $\Omega$ has compact support, provided $\omega + t\Omega$ is symplectic for all $t \in [0,1]$.
Suppose we also knew this to hold when $\Omega$ does not have compact support,
 and furthermore that the symplectic cohomology of $(\wh{\sbf(X)},\omega_{sf} + \e\Omega_{sf})$ is well-defined and invariant under symplectomorphisms. Then for small $\e > 0$ we would have
 \begin{align*}
\sh_{\e\Omega_{sf}} (\sbf(X^6)) &\cong \sh(\sbf(X^6),\omega_{sf} + \e\Omega_{sf}) \\&\cong \sh (X^6_+,\omega_+ +\e\Omega_+) \\&\cong \sh_{\e\Omega_+}(X^6_+).
\end{align*}
Moreover, based on the behavior of symplectic cohomology with respect to subcritical handles, it seems reasonable to guess that that latter term is isomorphic to $\sh(X^6)$.
This gives at least a heuristic explanation of Theorem \ref{thm:nontrivial twSH}.
 \end{remark}

To prove Theorem \ref{deformationthm}, we need the following lemma which is essentially due to Seidel.

\begin{lemma}\label{lem:squaredehntwist}
Let $(M^4,\omega)$ be a symplectic manifold, with $S \subset M$ a Lagrangian sphere and $\Omega$ a closed two-form on $M$ such that $[\Omega|_S] \neq 0 \in H^2(S;\R)$.
Consider the deformation of the symplectic form given by $\omega_s := \omega + s\Omega$.
Then for any neighborhood $U$ of $S$, there is a smooth family of symplectomorphisms $\Phi_s$ of $(M^4,\omega_s)$, defined for all $s \geq 0$ sufficiently small, such that
\begin{itemize}
\item
$\Phi_0 = \tau_S^2$
\item
$\Phi_s$ is supported in $U$ for all $s$
\item
for $s>0$, $\Phi_s$ is Hamiltonian isotopic to the identity by an isotopy supported in $U$.
\end{itemize}
\end{lemma}
\begin{proof}  
Let $(D^*_rS^2,\omega_{\std})$ denote the radius $r$ disk cotangent bundle of $S^2$ with respect to the round metric, for some $r$ sufficiently small.
Let $\Omega_{\std}$ denote the pullback to $D^*_rS^2$ of an $SO(3)$ invariant area form on $S^2$ with total area equal to $\int_S \Omega$.
Consider the deformation $(D^*_rS^2,\omega_{\std}+s\Omega_{\std})$ for $s \geq 0$ small.
By the usual Moser--Weinstein technique we can find a smooth family of symplectic embeddings 
\begin{align*}
E_s: (D^*_rS^2,\omega_{\std} + s\Omega_{\std}) \hookrightarrow (M^4,\omega_s)
\end{align*}
for all $s \geq 0$ sufficiently small.
By shrinking $r$ if necessary, we can ensure that all of these embeddings have image in $U$.

\begin{sloppypar} Now, following Seidel \cite{seidel2008lectures} there is a smooth family of symplectomorphisms $\Phi_s$  of ${(D^*_rS^2,\omega_{\std} + s\Omega_{\std})}$, defined for all $s \geq 0$ sufficiently small, such that \end{sloppypar}
\begin{itemize}
\item
$\Phi_0 = \tau_S^2$
\item 
$\Phi_s$ is supported in $\Int D^*_rS^2$
\item 
$\Phi_s$ is Hamiltonian isotopic to the identity by an isotopy supported in $\Int D^*_r S^2$.
\end{itemize}
Then the pushforward of $\Phi_s$ by $E_s$ has the desired properties.
\end{proof}

\sss

\noindent \textit{Proof of Theorem \ref{deformationthm}}:
Let $(M^4_{sf},\theta,\psi)$ denote the Weinstein fiber $\sbf(X^6)$, i.e. the result of attaching the Weinstein handles $H_1,...,H_k$ to the fiber of $X^6$.
We have symplectomorphisms
\begin{align*}
(\wh{\sbf(X)},\omega_{sf}) &\cong \wh{\fS}(M_{sf},d\theta;\tau_{S_1}^2V_1,...,\tau_{S_k}^2V_k)\\
(\wh{X_+},\omega_+) &\cong \wh{\fS}(M_{sf},d\theta;V_1,...,V_k).
\end{align*}
Let $\Omega$ be a closed two-form on $M_{sf}$ whose support is disjoint from $V_1 \cup ... \cup V_k \cup \tau_{S_1}^2V_1 \cup ... \cup \tau_{S_k}^2V_k$ and such that 
$[\Omega|_{S_i}] \neq 0 \in H^2(S_i;\R)$ for $1 \leq i \leq k$
and $\Omega$ is exact near $\bdy M_{sf}$ (we can take $\Omega$ to be Poincar\'e dual to the union of the cocores of the handles $H_1,...,H_k$).
As in $\S \ref{sec:Lefschetz structures}$,
we can find closed two-forms $\Omega_{sf}$ and $\Omega_+$ on $\wh{\sbf(X)}$ and $\wh{X_+}$ respectively such that 
for small $\e > 0$ we have symplectomorphisms
\begin{align*}
(\wh{\sbf(X)},\omega_{sf} + \e\Omega_{sf}) &\cong \wh{\fS}(M_{sf},d\theta + \e\Omega;\tau_{S_1}^2V_1,...,\tau_{S_k}^2V_k)\\
(\wh{X_+},\omega_+ + \e\Omega_+) &\cong \wh{\fS}(M_{sf},d\theta + \e\Omega;V_1,...,V_k).
\end{align*}
By Lemma \ref{lem:squaredehntwist}, $\tau_{S_i}^2V_i$ is Hamiltonian isotopic to $V_i$ in the symplectic manifold $(M_{sf},d\theta + \e\Omega)$ for $1 \leq i \leq k$.
By the Hamiltonian isotopy invariance of $\wh{S}$, it follows that 
$(\wh{\sbf(X)},\omega_{sf} + \e\Omega_{sf})$ and $(\wh{X_+},\omega_+ + \e\Omega_+)$ are symplectomorphic.
$\hfill\Box$

\section{Applications}\label{sec:applications}

We now make use of subflexibilization to construct exotic examples.
The starting point is Abouzaid--Seidel's \cite{abouzaid2010altering}  construction of an affine variety $U^{2n}$ for any $2n \geq 6$ such that $U$ is diffeomorphic to $\C^n$ and
$\sh(U) \neq 0$. Based on work of McLean \cite{mcleanlefschetz}, the idea is to take $U$ to be the Kaliman modification of $(\C^n,H,p)$, where $H \subset \C^n$ is a singular hypersurface given by the zero set of a weighted homogenous polynomial and $p \in H$ is a smooth point. This means that $U$ is obtained by blowing up $\C^n$ at $p$ and then excising the proper transform of $H$.
Choosing the weighted homogenous polynomial carefully, Abouzaid--Seidel show that $\C^n \setminus H$ has nonvanishing symplectic cohomology, and therefore so does $U$ by \cite[Theorem 2.31]{mcleanlefschetz}. Define $l_{0} \in \Z$ to be the minimal number of vanishing cycles of a Weinstein Lefschetz presentation for $U^{6}$ satisfying Assumption \ref{mainassumption}.

\sss

Before proving Theorem \ref{thm:advertising}, we briefly consider the effects of products, boundary connect sums, and subcritical handle attachments on subflexibility and twisted symplectic cohomology.

\begin{lemma}\label{lem:subflexprod}
Let $(W_1,\la_1,\phi_1)$ be a subflexible Weinstein domain and $(W_2,\la_2,\phi_2)$ any Weinstein domain. Then $(W_1,\la_1,\phi_1) \times (W_2,\la_2,\phi_2)$ is subflexible.
\end{lemma}
\begin{proof}
Up to a Weinstein homotopy we can find a flexible Weinstein domain $(X,\la,\phi)$ of which $(W_1,\la_1,\phi_1)$ is a sublevel set. Let $m_1$ and $m_2$ denote the maximal critical values of $\phi_1$ and $\phi_2$ respectively.
We can arrange that any critical points of $\phi$ on $X \setminus W_1$ have critical value at least $m$, for some $m > m_1 + m_2 - \min_{W_2}\phi_2$.
By Proposition \ref{prop:products are flexible}, $(\wh{X} \times \wh{W_2},\wh{\la} + \wh{\la_2},\wh{\phi} + \wh{\phi_2})$ is flexible, and moreover the sublevel set $\{\phi + \phi_2 \leq  C \}$ is Weinstein deformation equivalent to $W_1 \times W_2$ for any 
\begin{align*}
m_1 + m_2 < C <  m + \min_{W_2} \phi_2.
\end{align*}
\end{proof}
\begin{lemma}\label{lem:subflexsum}
Let $W_1$ and $W_2$ be subflexible Weinstein domains. Then $W_1 \bcs W_2$ is subflexible.
\end{lemma}
\begin{proof}
We can assume that $W_1$ and $W_2$ become flexible after attaching some number of Weinstein handles, whose attaching regions are disjoint from the attaching region of the $1$-handle joining $W_1$ and $W_2$ in the construction of $W_1 \bcs W_2$.
Therefore we can attach Weinstein handles to $W_1 \bcs W_2$ to obtain a boundary connect sum of two flexible Weinstein domains, and the latter is clearly flexible.
\end{proof}

\begin{prop}\label{prop:subcrit handle attach}
Let $(W_0,\la_0)$ be a Liouville domain, and let $(W,\la)$ be the resulting Liouville domain after attaching a subcritical Weinstein handle. Assume $c_1(W) = c_1(W_0) = 0$. Then for any closed two-form $\Omega$ on $W$, we have an isomorphism
\begin{align*}
\sh_{\Omega}^*(W,\la) \cong \sh_{\Omega|_{W_0}}^*(W_0,\la_0).
\end{align*}
\end{prop}
In the case of untwisted symplectic cohomology this is well-known result of Cieliebak \cite{handleattachinginSH}.
Following \cite[\S 10.3]{mcleanlefschetz}, the isomorphism is induced by the transfer map, which can be constructed using a sequence of admissible Hamiltonians in such a way that the subcritical handle introduces a single constant Hamiltonian orbit whose index is unbounded in the direct limit (here we need the Chern class assumption to define a $\Z$-grading).
In particular, one can easily check the the same proof applies in the presence of a twisting two-form.

\sss

\noindent \textit{Proof of Theorem \ref{thm:advertising}}:
We take $X'$ to be $\flex(X^{6}) \bcs \sbf(U^6)$ in the case $\dim X' = 6$, and for $2n > 6$ we set $$X' := \flex(X^{2n}) \bcs ( \sbf(U^6) \times D^*S^{2n-6})$$
(many other options are also possible, giving different subcritical topologies).
As explained in \S\ref{subsec:subflex}, $\sbf(U^6)$ is almost symplectomorphic to $U_+$, which is the result of attaching $l_0$ two-handles to the six-dimensional ball.
In particular, $\sbf(U^6)$ is almost symplectomorphic to a subcritical space, and hence so is $\sbf(U^6) \times D^*S^{2n-6}$.
Since $\flex(X)$ is almost symplectomorphic to $X$, this  proves the first part of the theorem.

By Theorem \ref{subflexthm}, $\sbf(U^6)$ becomes a ball after attaching $l_0$ Weinstein handles.
In fact, since $\D^2 \times D^*S^k$ becomes the ball after attaching a cancelling $(k+1)$-handle, we can also attach handles to convert $\sbf(U^6) \times D^*S^{2n-6}$ into the ball, and hence $X'$ into $\flex(X^{2n}) \bcs B^{2n} \cong \flex(X^{2n})$. This proves the second part of the theorem.

Finally, we endow $\sbf(U^6)$ with the twisting two-form from Theorem \ref{thm:nontrivial twSH} which makes the twisted symplectic cohomology of $\sbf(U^6)$ nontrivial. Extend this to a closed two-form on $X'$ by pulling it back under the projection $\sbf(U^6) \times D^*S^{2n-6} \rightarrow \sbf(U^6)$ and extending by zero over $\flex(X^{2n})$. 
By Proposition \ref{prop:subcrit handle attach} and Theorem \ref{thm:vanishing SH}, we can ignore the $\flex(X)$ factor for purposes of twisted symplectic cohomology.
The third part of the Theorem now follows by appealing to the K\"unneth Theorem \ref{thm:kunneth}, together with the well-known nonvanishing of symplectic cohomology for $D^*S^{2n-6}$.
$\hfill\Box$

\sss

We conclude by discussing how to use Theorem \ref{thm:advertising} to produce nonflexible polynomially convex domains, disproving the conjecture of Cieliebak--Eliashberg \cite{cieliebak2013topology}.
Following the conventions of \cite{cieliebak2013topology}, recall that a polynomially convex domain in $\C^n$ is a compact domain with smooth boundary $K \subset \C^n$ such that
$K$ coincides with its {\em polynomial hull}
\begin{align*}
\wh{K}_{\mathcal{P}} := \left\{ z \in \C^n \;:\; |P(z)| \leq \max_{u\in K}|P(u)| \text{ for all complex polynomials } P \text{ on } \C^n\right\}.
\end{align*}
(see also \cite{stout2007polynomial} for background on polynomial convexity).
According to Criterion 3.2 of \cite{cieliebak2013topology}, an $i$-convex domain $W \subset \C^n$ is polynomially convex if and only if there exists an exhausting $i$-convex function $\phi: \C^n \rightarrow \R$ for which $W$ is a sublevel set.
Here $W \subset \C^n$ is {\em i-convex} if there is an $i$-convex function on $W$ such that $\bdy W$ is its maximal regular level set.

\sss

\noindent \textit{Proof of Theorem \ref{thm:pc domains}}:
Since $X'$ is deformation equivalent to a sublevel set of $\flex(X)$, we focus on $\flex(X)$.
By \cite[Lemma 2.1]{cieliebak2013topology}, the smooth hypotheses on $X$ imply that the Morse function on $X$ extends to a Morse function on $\C^n$ without critical points of index greater than $n$. By \cite[Theorem 13.1]{cieliebak2012stein}, we can therefore extend the Weinstein structure on $X$ to a flexible Weinstein structure on $\C^n$, and by flexibility this is automatically Weinstein homotopic to the standard structure.  
Theorem 1.5(c) of \cite{cieliebak2013topology} now produces a Weinstein deformation equivalence of $X'$ onto a polynomially convex domain in $\C^n$.
$\hfill\Box$

\bibliographystyle{plain}
\bibliography{subflex}

\end{document}